\documentclass[11pt]{article}
\usepackage[english]{babel}
\usepackage{ae,lmodern}
\usepackage[T1]{fontenc}
\usepackage[utf8]{inputenc}
\usepackage{amsbsy}
\usepackage{amsmath}
\usepackage{amssymb}
\usepackage{amsfonts}
\usepackage{amsthm}
\usepackage{bm}
\usepackage{graphicx}
\usepackage{hyperref}
\usepackage[active]{srcltx}
\usepackage{upgreek}
\usepackage{xcolor}
\usepackage{hyperref}
\hypersetup{
	colorlinks   = true, %Colours links instead of ugly boxes
	urlcolor     = blue, %Colour for external hyperlinks
	linkcolor    = blue, %Colour of internal links
	citecolor   = red %Colour of citations
}

\newcommand{\C}{\mathbb{C}}

\newcommand{\N}{\mathbb{N}}

\newcommand{\R}{\mathbb{R}}

\newcommand{\T}{\mathbb{T}}

\newcommand{\Z}{\mathbb{Z}}

\newcommand{\boE}{\mathcal{E}}

\newcommand{\boI}{\mathcal{I}}
\newcommand{\boJ}{\mathcal{J}}
\newcommand{\boK}{\mathcal{K}}
\newcommand{\boL}{\mathcal{L}}

\newcommand{\boR}{\mathcal{R}}

\newcommand{\gR}{\mathfrak{R}}

\newcommand{\gc}{\mathfrak{c}}

\newcommand{\gp}{\mathfrak{p}}

\newcommand{\gu}{\mathfrak{u}}

\newtheorem{case}{Case}

\newtheorem{cor}{Corollary}
\newtheorem{lem}[cor]{Lemma}
\newtheorem{prop}[cor]{Proposition}
\newtheorem{thm}{Theorem}
\newtheorem{step}{Step}
\newtheorem*{thm*}{Theorem}
\newtheorem*{def*}{Definition}
\theoremstyle{definition}
\newtheorem*{merci}{Acknowledgments}

\theoremstyle{remark}

\begin{document}

\title{Construction of minimizing travelling waves for the Gross-Pitaevskii equation on $\R \times \T$}
\author{
\renewcommand{\thefootnote}{\arabic{footnote}}
Andr\'e de Laire\footnotemark[1],~Philippe Gravejat\footnotemark[2]~ and Didier Smets\footnotemark[3]}
\footnotetext[1]{
Univ.\ Lille, CNRS, Inria, UMR 8524 - Laboratoire Paul Painlev\'e, Inria, F-59000 Lille, France.
E-mail: {\tt andre.de-laire@univ-lille.fr}}
\footnotetext[2]{CY Cergy Paris Universit\'e, Laboratoire Analyse, G\'eom\'etrie, Mod\'elisation (UMR CNRS 8088), F-95302 Cergy-Pontoise, France. E-mail: {\tt philippe.gravejat@cyu.fr}}
\footnotetext[3]{Sorbonne Universit\'e, Laboratoire Jacques-Louis Lions (UMR CNRS 7598), F-75005 Paris, France. E-mail: {\tt didier.smets@sorbonne-universite.fr}}
\date{}
\maketitle

\begin{abstract}
As a sequel to our previous analysis in~\cite{deLGrSm1} on the Gross-Pitaevskii equation on the product space $\R \times \T$, we construct a branch of finite energy travelling waves as minimizers of the Ginzburg-Landau energy at fixed momentum. We deduce that minimizers are precisely the planar dark solitons when the length of the transverse direction is less than a critical value, and that they are genuinely two-dimensional solutions otherwise. The proof of the existence of minimizers is based on the compactness of minimizing sequences, relying on a new symmetrization argument that is well-suited to the periodic setting.
\end{abstract}

%%%%%%%%%%%%%%%%%%%%%%%%%%%%%%%%%%%%%%%%%%%%%%%%%%%%%%%%%%%%%%%%%

%%%%%%%%%%%%%%%%%%%%%%%%%%%%%%%%%%%%%%%%%%%%

%%%%%%%%%%%%%%
%%%%%%%%%%%%%%
%%%%%%%%%%%%%%
\section{Introduction}
%%%%%%%%%%%%%%
%%%%%%%%%%%%%%
%%%%%%%%%%%%%%

In this paper, we continue the study started in~\cite{deLGrSm1}, concerning the travelling wave solutions to the Gross-Pitaevskii equation
\begin{equation}
\label{eq:gp}
i \partial_t \Psi + \Delta \Psi + \Psi \big( 1 - |\Psi|^2\big) = 0
\end{equation}
on the product space $\R \times \T_\ell$, where $\T_\ell := \R / \ell \Z$ is the torus with fixed positive length $\ell$. In physics, this defocusing Schr\"odinger equation is a classical model for Bose-Einstein condensates, superfluidity, and nonlinear optical fibers~\cite{KevFrCa0, KivsLut1}. 

Let us recall that, in one space dimension, the Gross-Pitaevskii equation possesses a family of finite energy travelling waves, called dark solitons. They are given by the explicit formula
\begin{equation}
\label{eq:gu-val}
\gu_c(x) = \sqrt{\frac{2 - c^2}{2}} \tanh \bigg( \frac{\sqrt{2 - c^2}}{2} \, x \bigg) + i \frac{c}{\sqrt{2}},
\end{equation}
for any speed $|c| < \sqrt{2}$. These solitons extend trivially to the product space $\R \times \T_\ell$, where they are referred to as planar (or line) dark solitons. However, it is well-known in the physics literature that these planar solitons can be unstable due to the tendency to develop distortions in their transverse profile. We refer to~\cite{KevFrCa0, KuznTur1, GaidAng1} for details, and to~\cite{RousTzv3} for some rigorous results. In addition, experimental observations have shown that the dynamics of planar dark solitons are stable when they are sufficiently confined in the transverse direction, but unstable otherwise. In the latter case, the creation of vortices can occur (see~\cite{KivsPel1, HoefIla1} and the references therein).

In the sequel, we present a rigorous framework for studying this kind of phenomenon. Precisely, our goal is to prove the existence of non constant finite energy travelling wave solutions to~\eqref{eq:gp}, obtained as minimizers of the energy at fixed momentum. Taking into account the results in~\cite{deLGrSm1}, we will deduce that these minimizers are exactly the planar dark solitons when $\ell$ is less than a critical value, and that they are genuinely two-dimensional solutions otherwise. In particular, planar solitons do not minimize the energy in the presence of a large transverse direction.

To introduce our framework, we recall that it was proved in~\cite{BetGrSa2, BeGrSaS1} that the dark solitons
\eqref{eq:gu-val} are solutions to the constrained minimization problem 
\begin{equation}
\label{def:gI}
\boI_\textup{1d}(\gp) := \inf \big\{ E(\psi), \psi \in H_{\rm loc}^1(\R) \text{ s.t. } [P](\psi) = \gp \big\},
\end{equation}
for fixed $\gp \in \R / \pi \Z.$ Here the Ginzburg-Landau energy $E$ is defined as
$$
E(\psi) := \int_\R e(\psi) := \int_\R \bigg( \frac{1}{2} |\nabla \psi|^2 + \frac{1}{4} \big( 1 - |\psi|^2 \big)^2 \bigg),
$$
and $[P]$ is the (untwisted) momentum given by
$$
[P](\psi) := \frac{1}{2} \lim_{R \to + \infty} \bigg( \int_{- R}^{R} \langle i \partial_x \psi, \psi \rangle_\C + \textup{arg} \big( \psi(R) \big) - \textup{arg} \big( \psi(- R) \big) \bigg) \quad \text{ modulo } \pi. 
$$
The speed $c = c_\gp$ of the dark soliton $\gu_c$ is the Lagrange multiplier of this problem. It is uniquely determined by the identity $[P](\gu_{c_\gp}) = \gp$. The momentum $[P]$ is well-defined on the energy space
\begin{equation}
\label{def:X-space1d}
X(\R) := \big\{ \psi \in H_\text{loc}^1(\R) : \nabla \psi \in L^2(\R) \text{ and } 1 - |\psi|^2 \in L^2(\R) \big\}.
\end{equation}
This claim was proved in~\cite{BeGrSaS1}, together with the fact that the definition of $[P]$ only makes sense modulo $\pi.$ 

We now turn our attention to the corresponding minimization problem on the product space $\R \times \T_\ell$. We normalize the Ginzburg-Landau energy as
\begin{equation}
\label{def:E-energy}
E(\psi) := \frac{1}{\ell}\int_{\R} \int_{\T_\ell} e(\psi),
\end{equation}
so that functions, which only depend on the horizontal variable, have the same energy values in one space and in two space dimensions. Given a number $\gp \in \R / \pi \Z,$ we set
\begin{equation}
\label{def:I-l}
\boI_\textup{2d}(\gp) := \inf \big\{ E(\psi) : \psi \in X(\R \times \T_\ell) \text{ with } [P](\psi) = \gp \big\}.
\end{equation}
Here the natural energy space is defined as above by
$$
X(\R \times \T_\ell) = \big\{\psi \in H_\text{loc}^1(\R \times \T_\ell) : \nabla \psi \in L^2(\R \times \T_\ell) \text{ and } 1 - |\psi|^2 \in L^2(\R \times \T_\ell) \big\}.
$$
The (untwisted) momentum $[P]$ requires some proper definition. For that purpose, we decompose a function $\psi \in X(\R \times \T_\ell)$ as
\begin{equation}
\label{eq:decomp1d}
\psi(x,y) = \widehat{\psi}(x) + w(x,y),
\end{equation}
where
$$
\widehat{\psi}(x) := \frac{1}{\ell}\int_{\T_\ell} \psi(x,y) \, dy.
$$
We recall from~\cite{deLGrSm1} that $\widehat{\psi} \in X(\R)$, and $w \in H^1(\R \times \T_\ell)$, so that we may define
\begin{equation}
\label{eq:defP2D}
[P](\psi) := [P](\widehat{\psi}) + \frac{1}{2\ell} \int_\R \int_{\T_\ell} \langle i \partial_x w, w \rangle_\C \quad \text{ modulo } \pi.
\end{equation}
Note that here also, if $\psi$ only depends on the horizontal variable, the versions of $[P]$ in one and two space dimensions coincide, which justifies our slight abuse of notation.

In~\cite{deLGrSm1}, we established that the problem $\boI_\textup{2d}(\gp)$ is achieved by one-dimensional minimizers, the planar dark solitons, when $\ell$ is sufficiently small, but that no minimizer can be one dimensional if $\ell$ is sufficiently large. In the present paper, we prove that minimizers do exist for all possible values of $\ell$ and $\gp$.

\begin{thm}
\label{thm:exist-min}
Let $\ell > 0 $ and $\gp \in \R / \pi \Z$. The minimization problem $\boI_{\textup{2d}}(\gp)$ is achieved by some function $U_{\gp, \ell} \in X(\R \times \T_\ell)$. Moreover, $U_{\gp, \ell}$ is smooth on $\R \times \T_\ell$, and there exists a number $c_{\gp, \ell} \in \R$ such that
\begin{equation}
\label{eq:Psi_p}
i c_{\gp, \ell} \partial_x U_{\gp, \ell} + \Delta U_{\gp, \ell} + U_{\gp, \ell} \, \big( 1 - |U_{\gp, \ell}|^2 \big) = 0.
\end{equation}
\end{thm}

Combining with~\cite[Theorem 1]{deLGrSm1}, we conclude that, for a given $\gp \in \R / \pi \Z$, there is a critical length $\ell_\gp > 0$, such that, if $\ell \in (0, \ell_\gp)$, then $U_{\gp, \ell}$ is a planar soliton, up to a translation and a phase shift. On the other hand, if $\ell\in (\ell_\gp, \infty)$, then $U_{\gp, \ell}$ cannot solely depend only on the $x$ variable, and therefore cannot be a planar soliton.

Existence of travelling waves for the Gross-Pitaevskii equation has attracted a lot of efforts in the case of the whole plane $\R^2$. In~\cite{BethSau1}, the existence of non-constant finite energy travelling waves with arbitrary small speeds $c$ was proved using some mountain pass argument. The associated momentum behaves like $p \simeq \log(1/c)$ in the limit $c \to 0$. In~\cite{BetGrSa1}, the previous result was extended to the case of travelling waves with an arbitrary value for the momentum $p$. These solutions are also minimizers of the energy at fixed momentum, but they were obtained as local limits of minimizers on expanding tori. Existence and compactness for minimizing sequences directly on $\R^2$ was proved in~\cite{ChirMar1}, together with the stability of the minimizing set. In these last two works, little information is given regarding the speed of the travelling waves, beyond the fact that they are subsonic, i.e.\ $|c| < \sqrt{2}$. The same limitation holds in Theorem~\ref{thm:exist-min} above. In~\cite{BellRui1}, existence of travelling waves in $\R^2$ for almost every value of $|c| < \sqrt{2}$ was proved, using a mountain pass approach combined with a monotonicity argument. The question of existence for the full range of speeds $|c| < \sqrt{2}$ remains open, both in the case of $\R^2$ and of $\R \times \T_\ell.$ In contrast, for the torus $\T_\ell \times \T_\ell$ it is known from~\cite{MaSaRui1} that travelling waves exist with arbitrary speeds. Only these with speed $|c| < \sqrt{2}$ could possibly converge to finite energy travelling waves in $\R^2$ by a limiting procedure.

In the papers quoted so far, as well as in Theorem~\ref{thm:exist-min}, the uniqueness of minimizers (up to the geometric invariances) is not tackled. In the case of $\R^2$, this was recently proved in~\cite{ChirPac3} when the speed $c$ is small enough, using a delicate perturbative argument.

We finally mention that stationary solutions of the Gross-Pitaevskii equation in the strip $\R \times [-\ell, \ell]$ have recently been constructed in~\cite{AftaSan1}. It is tempting to relate them to the ones obtained in Theorem~\ref{thm:exist-min} for the special case $\gp = \frac{\pi}{2}$ modulo $\pi$. For small values of $\ell$, these are the one-dimensional black soliton, which is stationary, while for large values of $\ell$, they are truly two-dimensional.

In Section~\ref{sec:sketch}, we present the main ideas for the proof of Theorem~\ref{thm:exist-min} as well as the statements of the intermediate results. The proofs for the latter are then given in the remaining sections.

{\it Notations:} In many places, we shall write $\gp \in \R / \pi \Z$ as $\gp = [p]$ for some representative value $p \in \R$. Whenever a function $f$ is defined on $\R / \pi \Z$, we often identify it with a $\pi$-periodic function on $\R$, and write $f(p)$ in place of $f([p])$. We also use the notation $|\gp|$ to refer to the distance between $\gp$ and $[0]$ in $\R / \pi\Z$.

%%%%%%%%%%%%%%%%%%%%%%%%%%%%%%%%%%%%%
%%%%%%%%%%%%%%%%%%%%%%%%%%%%%%%%%%%%%
%%%%%%%%%%%%%%%%%%%%%%%%%%%%%%%%%%%%%
\section{Sketch of the proof of Theorem~\ref{thm:exist-min}}
\label{sec:sketch}
%%%%%%%%%%%%%%%%%%%%%%%%%%%%%%%%%%%%%
%%%%%%%%%%%%%%%%%%%%%%%%%%%%%%%%%%%%%
%%%%%%%%%%%%%%%%%%%%%%%%%%%%%%%%%%%%%

Our goal is to establish some form of compactness for minimizing sequences of the problem $\boI_\textup{2d}(\gp)$. For that purpose, we rely on a concentration-compactness type argument in the locally compact case. In Subsection~\ref{sub:energy}, we analyze the function $\gp \mapsto \boI_{\textup{2d}}(\gp)$ and in particular prove its strict sub-additivity. For that purpose, we introduce a slicing and mirroring argument, which we believe is original and well-adapted to the periodic setting. In Subsection~\ref{sub:conc-comp}, we finally obtain compactness by excluding first vanishing and then dichotomy for minimizing sequences.

%%%%%%%%%%%%%%%%%%%%%%%%%%%%%%
%%%%%%%%%%%%%%%%%%%%%%%%%%%%%%
\subsection{Properties of the minimizing energy}
\label{sub:energy}
%%%%%%%%%%%%%%%%%%%%%%%%%%%%%%
%%%%%%%%%%%%%%%%%%%%%%%%%%%%%%

The first step is to analyze precisely the properties of the minimizing energy $\boI_\textup{2d}(\gp)$, in particular its possible strict sub-additivity with respect to $\gp$. In~\cite[Lemma 2]{deLGrSm1}, we already derived some preliminary properties of this minimizing energy, which we now recall for the sake of completeness.

\begin{lem}
\label{lem:prop-I-l}
Let $\ell > 0$. The function $\boI_\textup{2d}$ is an even and Lipschitz continuous function on $\R / \pi \Z$, with Lipschitz constant at most $\sqrt{2}$, and it is bounded by
\begin{equation}
\label{eq:estim-sup-I}
\boI_\textup{2d}(\gp) \leq \boI_\textup{1d}(\gp) < \sqrt{2} \, |\gp|,
\end{equation}
for any $\gp \in \R / \pi \Z$ (with $\gp \neq [0]$ in the last inequality).
\end{lem}

The main properties of the minimizing energy $\boI_\textup{2d}$ can be summarized as follows.

\begin{prop}
\label{prop:prop-I-l}
Let $\ell > 0$. The map $p \mapsto \boI_\textup{2d}(p)$ is concave on $[0, \pi]$. Moreover, the function $\boI_\textup{2d}$ is strictly sub-additive on $\R / \pi \Z$, in the sense that
\begin{equation}
\label{eq:str-sub-add}
\boI_\textup{2d}(\gp_1 + \gp_2) < \boI_\textup{2d}(\gp_1) + \boI_\textup{2d}(\gp_2),
\end{equation}
for any $(\gp_1, \gp_2) \in (\R / \pi \Z)^2$ such that $\gp_1 \neq [0]$ and $\gp_2 \neq [0]$.
\end{prop}

The proof of Proposition~\ref{prop:prop-I-l} is reminiscent from the description of a similar minimizing energy in~\cite[Section 3]{BetGrSa1}. This description is heavily based on symmetrization arguments. Due to our periodic setting in the transverse direction $y$, we cannot invoke directly the arguments in~\cite{BetGrSa1}, and we have to refine them properly in order to establish Proposition~\ref{prop:prop-I-l}. 

Concerning the concavity of the map $\boI_\textup{2d}$ on the interval $[0, \pi]$, we rely on the following characterization of concave functions.

\begin{lem}
\label{lem:concave}
Let $f : [a, b] \to \R$ be a continuous function. Given any number $x \in (a, b)$, assume the existence of a number $\delta_x > 0$ such that
\begin{equation}
\label{eq:cond-concave}
\frac{1}{2} \Big( f(x + \delta) + f(x - \delta) \Big) \leq f(x),
\end{equation}
for any number $0 \leq \delta \leq \delta_x$. Then, the function $f$ is concave on $[a, b]$. 
\end{lem}

It is well-known that a continuous function $f$ satisfying the inequalities
$$
\frac{1}{2} \Big( f(x_1) + f(x_2) \Big) \leq f \Big( \frac{x_1 + x_2}{2} \Big),
$$
for any numbers $a \leq x_1 \leq x_2 \leq b$ is concave. Lemma~\ref{lem:concave} extends this classical characterization of concavity. For the sake of completeness, we detail the proof of Lemma~\ref{lem:concave} in Appendix~\ref{sec:concave} below.

We are able to check that the function $\boI_\textup{2d}$ satisfies the condition in~\eqref{eq:cond-concave} on the interval $[0, \pi]$.

\begin{lem}
\label{lem:cond-concave}
Let $0 < p < \pi$. There exists a number $\delta_p > 0$ such that
\begin{equation}
\label{eq:cond-concave-I-l}
\frac{1}{2} \Big( \boI_\textup{2d}(p + \delta) + \boI_\textup{2d}(p - \delta) \Big) \leq \boI_\textup{2d}(p),
\end{equation}
for any number $0 \leq \delta \leq \delta_p$.
\end{lem}

For the proof of Lemma~\ref{lem:cond-concave}, we distinguish two situations.

\begin{case}
\label{Case1}
$\boI_\textup{2d}(p) = \boI_\textup{1d}(p)$.
\end{case}

In this situation, the condition in~\eqref{eq:cond-concave-I-l} is a direct consequence of the concavity of the function $\boI_\textup{1d}$ on $[0, \pi]$ (see Lemma~\ref{lem:prop-gI}). We can indeed use this property to exhibit a number $\delta_p > 0$ such that
$$
\frac{1}{2} \Big( \boI_\textup{1d}(p + \delta) + \boI_\textup{1d}(p - \delta) \Big) \leq \boI_\textup{1d}(p),
$$
for any $0 \leq \delta \leq \delta_p$. Invoking~\eqref{eq:estim-sup-I}, we obtain
$$
\frac{1}{2} \Big( \boI_\textup{2d}(p + \delta) + \boI_\textup{2d}(p - \delta) \Big) \leq \boI_\textup{1d}(p) = \boI_\textup{2d}(p).
$$
This is exactly~\eqref{eq:cond-concave-I-l}.

\begin{case}
\label{Case2}
$\boI_\textup{2d}(p) < \boI_\textup{1d}(p)$.
\end{case}

Given a minimizing sequence for $\boI_\textup{2d}(p)$ and a number $\delta > 0$, we shall construct and evaluate modified minimizing sequences for $\boI_\textup{2d}(p\pm \delta)$ using a slicing and mirroring argument with respect to a dyadic decomposition in the variable $y$. For that purpose, we first invoke the Fubini theorem to write 
$$
E(\psi) = \frac{1}{\ell} \int_{\T_\ell} E \big( \psi( \cdot, y) \big) \, dy \geq \frac{1}{\ell} \int_{\T_\ell} \boI_\textup{1d} \big( [P](\psi(\cdot, y)) \big) \, dy,
$$
for any function $\psi \in X(\R \times \T_\ell)$. Since $\boI_\textup{1d}$ is $\sqrt{2}$-Lipschitz on $\R / \pi \Z$ (see Lemma~\ref{lem:prop-gI}), it follows that
$$
E(\psi) - \boI_\textup{1d} \big( [P](\psi) \big) \geq - \frac{\sqrt{2}}{\ell }\int_{\T_\ell} \big| [P](\psi(\cdot, y)) - [P](\psi) \big| \, dy.
$$
We have hence proved

\begin{lem}
\label{lem:not-1d}
Let $\psi \in X(\R \times \T_\ell)$ such that $[P](\psi) = [p] \in \R/\pi \Z$. We have
\begin{equation}
\label{eq:not-1d}
\frac{1}{\ell} \int_{\T_\ell} \big| [P](\psi(\cdot, y)) - [p] \big| \, dy \geq \frac{1}{\sqrt{2}} \big( \boI_\textup{1d}(p) - E(\psi) \big). 
\end{equation}
\end{lem}

Consider now a minimizing sequence $(\psi_n)_{n \geq 0}$ for $\boI_\textup{2d}(p)$. Provided that $n$ is sufficiently large, it satisfies
\begin{equation}\label{eq:defsigma}
\frac{1}{\ell} \int_{\T_\ell} \big| [P](\psi_n(\cdot, y) - [p] \big| \, dy \geq \frac{1}{2} \big( \boI_\textup{1d}(p) - \boI_\textup{2d}(p) \big) > 0,
\end{equation}
so that the numbers $[P](\psi_n(\cdot, y))$ cannot be uniformly close to $[p]$ in the limit $n \to \infty$. For $j \geq 0$, we then define the quantities
\begin{equation}
\label{def:delta-n-j}
\delta^{(j)} := \limsup_{n \to \infty } \, \sup_{I_j} \bigg| \frac{1}{|I_j|} \int_{I_j} [P] \big( \psi_n(\cdot, y) \big) \, dy - [p] \bigg|,
\end{equation}
where $I_j$ refers to an arbitrary interval of size $2^{- j} \, \ell$ in $\T_\ell$. 
Although $[P](\psi_n(\cdot,y))$ only makes sense as a number in $\R / \pi\Z$, note that it is equal to
\begin{equation}
\label{eq:cle}
[P] \big( \psi_n(\cdot, y) \big) = [P] \big( \widehat{\psi}_n \big) + \frac{1}{2} \int_\R \langle i \partial_x w_n(x, y), w_n(x, y) \rangle_\C \, dx,
\end{equation}
and therefore of the form of a fixed value in $\R / \pi \Z$ added to an integrable real-valued function of $y$. This provides the correct meaning to the integral term in~\eqref{def:delta-n-j}.

Using oscillation estimates in the variable $y$, we next derive from~\eqref{eq:not-1d} that

\begin{lem}
\label{lem:presymmetry}
There exist numbers $n_0 \in \N$, $\delta_0 > 0$ and $j_0 \in \N$ such that, for any $n \geq n_0$, any $0 \leq |\delta| \leq \delta_0$, and any $j \geq j_0$, there exists an interval $I_n$ of size $2^{- j} \ell$ in $\T_\ell$ such that the sequence $(\psi_n)_{n \geq 0}$ satisfies
$$
\frac{1}{|I_n|} \int_{I_n} [P] \big( \psi_n(\cdot, y) \big) \, dy = [p + \delta].
$$
\end{lem}

We are now in position to finish the proof of Lemma~\ref{lem:cond-concave} in Case~\ref{Case2}, using a symmetrization argument. 

First, it follows from Lemma~\ref{lem:presymmetry} that $\delta^{(j)} > 0$ for any $j \geq j_0.$ We next define
$$
j_* := \min \left\{ j \geq 0,\text{ s.t. } \delta^{(j)} > 0\right\}. 
$$

Set $h_* := 2^{- j_*} \, \ell$ and $\delta_* = \delta^{(j_*)}/2$. By definition of $j_*$, there exists a subsequence (still denoted $\psi_n$) for which, for each $n$, there exists an interval $I_n$ of size $h_*$ such that 
$$
\frac{1}{|I_n|} \int_{I_n} [P] \big( \psi_n(\cdot, y) \big) \, dy = [ p + \epsilon \delta_* ],
$$
with $\epsilon = \pm 1$. Since the function $t \mapsto \frac{1}{h_*} \int_{t}^{t+h_*} [P] \big( \psi_n(\cdot, y) \big) \, dy$ is continuous and of mean value $[p]$, for any $\delta$ in between $0$ and $\epsilon \delta_*$, and for any $n$, there exist numbers $y_n$ such that
$$
\frac{1}{h_*} \int_{y_n}^{y_n + h_*} [P] \big( \psi_n(\cdot, y) \big) \, dy = [p + \delta].
$$
Without loss of generality, we can use translation invariance in order to assume that $y_n = 0$. We next consider the $2^{j_*}$ subintervals $J_k = [k \, h_*, (k + 1) \, h_*]$ of $[0, \ell]$. By construction, we already have
$$
\frac{1}{|J_0|} \int_{J_0} [P] \big( \psi_n(\cdot, y) \big) \, dy = [p + \delta].
$$
By definition of the integer $j_*$, and in particular by the fact that $\delta^{(j_* - 1)} = 0$, we deduce that
$$
\lim_{n \to \infty} \frac{1}{2|J_k|} \int_{J_k \cup J_{k+(-1)^{k+1}}} [P] \big(
\psi_n(\cdot, y) \big) \, dy = [p],
$$
and therefore
$$
\lim_{n \to \infty} \frac{1}{|J_k|} \int_{J_k} [P] \big( \psi_n(\cdot, y) \big) \, dy = [p + (- 1)^k \delta].
$$
For each $0 \leq k \leq 2^{j_*} - 1$, we define the symmetrized function $\psi_{n, k}$ by the properties :
\begin{enumerate}
\item $\psi_{n, k}$ is equal to $\psi_n$ on the strip $\R \times J_k$,
\item $\psi_{n, k}$ is mirror symmetric between the two consecutive strips $\R \times J_k$ and $\R \times J_{k + 1}$,
\item $\psi_{n, k}$ is invariant by translation of $2 h_*$ in the variable $y$.
\end{enumerate}
Note that by construction the resulting functions $\psi_{n, k}$ still belong to $X(\R \times \T_\ell)$. Their energy is given by
$$
E \big( \psi_{n, k} \big) = \frac{1}{h_*} \int_{J_k} \int_\R e(\psi_n)(x, y) \, dx \, dy,
$$
and we also have
\begin{equation}
\label{def:p-n-m}
[P] \big( \psi_{n, k} \big) = \frac{1}{h_*} \int_{J_k} [P] \big( \psi_n(x, y) \big) \, dx \, dy \to \big[ p + (- 1)^k \delta \big],
\end{equation}
as $n \to \infty$. As a consequence of the definition of the minimization problem in~\eqref{def:I-l}, these energies satisfy
\begin{equation}
\label{eq:conc-ineq-n-m}
2^{j_*} E \big( \psi_n \big) = \sum_{k = 0}^{2^{j_*} - 1} E \big( \psi_{n, k} \big) \geq \sum_{k = 0}^{2^{j_*} - 1} \boI_\textup{2d} \big( [P](\psi_{n, k}) \big).
\end{equation}
Combining~\eqref{eq:conc-ineq-n-m} with~\eqref{def:p-n-m}, and the continuity of the minimizing energy $\boI_\textup{2d}$ given by Lemma~\ref{lem:prop-I-l}, we are led to the inequality
$$
\boI_\textup{2d}(p) \geq \frac{1}{2} \Big( \boI_\textup{2d}(p + \delta) + \boI_\textup{2d}(p - \delta) \Big),
$$
in the limit $n \to \infty$. This is exactly the inequality in~\eqref{eq:cond-concave-I-l}, and this inequality thus holds for $|\delta| \leq \delta_p := \delta_*$. This completes the proof of Lemma~\ref{lem:cond-concave} in Case~\ref{Case2}.

In order to complete the proof of Proposition~\ref{prop:prop-I-l}, it remains to derive the strict sub-additivity of the minimizing energy $\boI_\textup{2d}$. This property results from the concavity of this function by invoking the following lemma. 

\begin{lem}
\label{lem:strict-subadditivity}
Given a number $R > 0$, consider a non-negative and concave function $f$ on $[0, R]$ such that $f(0) = 0$. Either the function $f$ is strictly sub-additive on $[0, R]$, or there exist two numbers $0 < x_* \leq R$ and $\mu \geq 0$ such that
\begin{equation}
\label{eq:linear}
f(x) = \mu x,
\end{equation}
for $0 \leq x \leq x_*$.
\end{lem}

When the function $f$ is concave on $[0, R]$, with $f(0) = 0$, we indeed know that
$$
f(y) \geq \frac{y}{x} f(x) + \Big( 1 - \frac{y}{x} \Big) f(0) \geq \frac{y}{x} f(x), 
$$
for any $0 < y \leq x \leq R$. Summing these inequalities for $0 < y_1, y_2 \leq R$ such that $0 < y_1 + y_2 \leq R$, gives
$$
f(y_1) + f(y_2) \geq \frac{y_1}{y_1 + y_2} f(y_1 + y_2) + \frac{y_2}{y_1 + y_2} f(y_1 + y_2) = f(y_1 + y_2),
$$
that is the sub-additivity of the function $f$. The alternative in Lemma~\ref{lem:strict-subadditivity} follows from analyzing the possible case of equality in this inequality. We refer to Appendix~\ref{sec:concave} below for more details. 

The function $\boI_\textup{2d}$ cannot be linear in a right neighborhood of zero, since on the one hand by inequality~\eqref{eq:estim-sup-I}, we have $\boI_\textup{2d}(p) < \sqrt{2} \, p$, and on the other hand, we have

\begin{lem}
\label{lem:deriv-0}
The function $\boI_\textup{2d}$ satisfies $\boI_\textup{2d}(p) \sim \sqrt{2} \, p$ as $p \to 0^+$.
\end{lem}

Therefore the minimizing energy $\boI_\textup{2d}$ is strictly sub-additive on $[0, \pi]$. The strict sub-additivity on $\R /\pi \Z$ follows from parity and periodicity arguments. We refer to Section~\ref{sec:energy} for the details, and mention that the proof of Lemma~\ref{lem:deriv-0} relies on arguments similar to the ones that we are now going to describe in the next subsection and which are related to the possible vanishing of minimizing sequences for~\eqref{def:I-l}.

%%%%%%%%%%%%%%%%%%%%%%%%%%%%%%%%%
%%%%%%%%%%%%%%%%%%%%%%%%%%%%%%%%%
\subsection{The concentration-compactness argument}
\label{sub:conc-comp}
%%%%%%%%%%%%%%%%%%%%%%%%%%%%%%%%%
%%%%%%%%%%%%%%%%%%%%%%%%%%%%%%%%%

We fix a number $\gp = [p] \in \R / \pi \Z$, with $[p] \neq [0]$. In case $[p] = [0]$, the minimization problem $\boI_\textup{2d}(0)$ is indeed achieved by the constant functions $\psi$ of modulus one for which $[P](\psi) = [0]$, and $\boI_\textup{2d}(0) = E(\psi) = 0$. We consider a minimizing sequence of functions $(\psi_n)_{n \geq 0}$ for~\eqref{def:I-l}, which we write as $\psi_n = \widehat{\psi}_n + w_n$ according to~\eqref{eq:decomp1d}.
We wish to show some compactness for the sequence $(\psi_n)_{n \geq 0}$, up to possible translations and constant
phase shifts. The main obstacle is the unboundedness of the domain in the $x$-direction. We deal with it using the
classical scheme of concentration-compactness, by establishing a non-vanishing and a non-splitting property.

\begin{def*}
We say that a sequence $(\psi_n)_{n \geq 0}$ in $X(\R \times \T_\ell)$ is vanishing if for some $r > 0$,
\begin{equation}
\label{eq:vanishing}
\liminf_{n \to \infty} \, \sup_{a \in \R} \, \frac{1}{\ell} \int_{B(a,r)} \int_{\T_\ell} e(\psi_n) = 0.
\end{equation}
\end{def*}

We first prove

\begin{lem}
\label{lem:conv-const}
Let $(\psi_n)_{n \geq 0}$ in $X(\R \times \T_\ell)$ be a vanishing sequence. Then 
writing $\psi_n = \widehat{\psi}_n + w_n$ as in~\eqref{eq:decomp1d}, we have
$$
\liminf_{n \to \infty} \, \sup_{x \in \R} \, \big| 1 - |\widehat{\psi}_n(x)| \big| = 0.
$$
\end{lem}

In the sequel, we make use of the following decomposition based on the splitting in~\eqref{eq:decomp1d}
\begin{equation}
\label{eq:decompEw}
E(\psi_n) = E(\widehat{\psi}_n) + \frac{1}{2 \ell} \int_\R \int_{\T_\ell} \bigg( |\nabla w_n|^2 + 2 \langle w_n, \widehat{\psi}_n
\rangle_\C^2 - |w_n|^2 \big( 1 - |\widehat{\psi}_n|^2 \big) + 2 \langle w_n, \widehat{\psi}_n \rangle_\C \, |w_n|^2 + \frac{1}{2} |w_n|^4 \bigg).
\end{equation}
For a vanishing sequence, it is expected that the super-quadratic terms in the right-hand side of~\eqref{eq:decompEw} tend to $0$ in the limit $n \to \infty$. Invoking Lemma~\ref{lem:conv-const} we can similarly simplify the expression for the momenta $[P](\psi_n)$ in~\eqref{eq:defP2D}. More precisely, we prove 

\begin{lem}
\label{lem:lions}
Let $(\psi_n)_{n \geq 0}$ in $X(\R \times \T_\ell)$ be a vanishing sequence such that $\sup_{n \geq 0} E(\psi_n) < + \infty.$
Up to a subsequence, we have
\begin{equation}
\label{eq:retraite}
E(\psi_n) = E(\widehat{\psi}_n) + \frac{1}{2 \ell} \int_\R \int_{\T_\ell} \Big( |\nabla w_n|^2 + 2 \langle w_n,
\widehat{\psi}_n \rangle_\C^2 \Big) + o(1),
\end{equation}
and in~\eqref{eq:defP2D},
\begin{equation}
\label{eq:tauxplein}
\frac{1}{2 \ell} \int_\R \int_{\T_\ell} \langle i \partial_x w_n, w_n \rangle_\C = \frac{1}{\ell} \int_\R \int_{\T_\ell}
\frac{1}{|\widehat{\psi}_n|^2} \, \langle w_n, \widehat{\psi}_n \rangle_\C \, \langle i \partial_x w_n , \widehat{\psi}_n
\rangle_\C + o(1),
\end{equation}
as $n \to \infty$.
\end{lem}

Inspection of the last integral term in~\eqref{eq:tauxplein} shows that it can be controlled by the last one
in~\eqref{eq:retraite} provided that the modulus $|\widehat{\psi}_n|$ is controlled from below. A similar property
holds for the term $[P](\widehat{\psi}_n)$ in~\eqref{eq:defP2D}, which can be bounded by the energy
$E(\widehat{\psi}_n)$, under a similar control on the modulus $|\widehat{\psi}_n|$. Combined, these lead to

\begin{lem}
\label{lem:farfromone}
Let $(\psi_n)_{n \geq 0}$ in $X(\R \times \T_\ell)$ be a vanishing sequence such that $\sup_{n \geq 0} E(\psi_n) < + \infty$. Up to a subsequence we have 
$$
\delta_n := \inf_{x \in \R} |\widehat{\psi}_n(x)| > 0,
$$
and 
\begin{equation}
\label{eq:contre-64}
\big| [P](\psi_n) \big| \leq \frac{E(\psi_n)}{\sqrt{2} \, \delta_n} + o(1),
\end{equation}
as $n \to \infty$.
\end{lem}

We are now in position to state and prove

\begin{prop}
\label{prop:nonvanishing}
A minimizing sequence $(\psi_n)_{n \geq 0}$ for $\boI_\textup{2d}(p)$ cannot be vanishing.
\end{prop}

Indeed, assume by contradiction that it is vanishing, and set $\boE = \liminf_{n \to \infty} E(\psi_n)$ (which is finite since $\psi_n$ is assumed to be minimizing). From Lemma~\ref{lem:conv-const} we derive that
$$
\delta_n := \inf_{x \in \R} |\widehat{\psi}_n(x)| \to 1,
$$
as $n \to \infty$. Moreover, we can assume, up to a subsequence, that the energies $E(\psi_n)$ tend to $\boE$, so that they are uniformly bounded. Therefore, we deduce from the fact that $[P](\psi_n) = [p]$,~\eqref{eq:estim-sup-I} and~\eqref{eq:contre-64} that
$$
\boI_\textup{2d}(p) < \sqrt{2} \, \big| [p] \big| \leq \boE.
$$
This is in contradiction with the fact that $(\psi_n)_{n \geq 0}$ is minimizing, and therefore completes the proof of Proposition~\ref{prop:nonvanishing}.

The non-vanishing property allows us to extract a limiting profile $\psi_\infty$ after a suitable translation (we will eventually prove it is non-constant). By Proposition~\ref{prop:nonvanishing}, and the invariance of the energy and the momentum under translation, we may assume that there exists $\delta_\infty > 0$ such that up to a subsequence
\begin{equation}
\label{eq:massezero}
\frac{1}{\ell} \int_{B(0, 1)} \int_{\T_\ell} e(\psi_n) \geq \delta_\infty > 0, 
\end{equation}
for any $n \geq 0.$ By boundedness of the energies $E(\psi_n)$, we can extract a further subsequence and a limiting profile $\psi_\infty \in X(\R \times \T_\ell)$ such that 
\begin{equation}
\label{eq:convfaibleprofil}
\nabla \psi_n \rightharpoonup \nabla \psi_\infty, \quad \text{ and } \quad 1 - |\psi_n|^2 \rightharpoonup 1 - |\psi_\infty|^2 \quad \text{ in } L^2(\R \times \T_\ell), 
\end{equation}
and
\begin{equation}
\label{eq:convfortprofil}
\psi_n \to \psi_\infty \quad \text{ in } L_\text{loc}^p(\R \times \T_\ell), 
\end{equation}
for any $1 \leq p < + \infty$. We decompose once more the functions $\psi_n$ and $\psi_\infty$ as
$$
\psi_n = \widehat{\psi}_n + w_n, \quad \text{ and } \quad \psi_\infty = \widehat{\psi}_\infty + w_\infty.
$$
By the Sobolev embedding theorem, we know that $\widehat{\psi}_n \to \widehat{\psi}_\infty$ locally uniformly on $\R$. In a concentration-compactness framework, it would then be classical to consider the difference $\widetilde{\psi}_n := \psi_n - \psi_\infty$ and to try to prove that
$$
E(\psi_n) \geq E(\psi_\infty) + E(\widetilde{\psi}_n) + o(1), 
$$
and
$$
[P](\psi_n) = [P](\psi_\infty) + [P](\widetilde{\psi}_n) + o(1),
$$
as $n \to \infty$. We cannot argue exactly this way because our functional space is not a vector space and the quantities $E(\widetilde{\psi}_n)$ and $[P](\widetilde{\psi}_n)$ simply do not make any sense. For $m \geq 1$, we shall instead construct functions $\widetilde{\psi}_n$ in the energy space, which will correspond to $\psi_n$ without its localized bump.

In the sequel, given a number $R > 0$ and a function $\psi = \widehat{\psi} + w \in X(\R \times \T_\ell)$, we shall use the notations
\begin{equation}
\label{def:E_R}
E_R \big( \psi \big) = \frac{1}{\ell} \int_{- R}^R \int_{\T_\ell} e(\psi)(x, y) \, dx \, dy, 
\end{equation}
and, provided that $\widehat{\psi}(\pm R) \neq 0$,
\begin{equation}
\label{def:P_R}
[P_R] \big( \psi \big) = [P_R] \big( \widehat{\psi} \big) + \frac{1}{2 \ell} \int_{- R}^R \int_{\T_\ell} \big\langle i \partial_x w(x, y), w(x, y) \big\rangle_\C \, dx \, dy,
\end{equation}
where
$$
[P_R] \big( \widehat{\psi} \big) := \frac{1}{2} \int_{- R}^R \big\langle i \partial_x \widehat{\psi}, \widehat{\psi} \big\rangle_\C + \frac{1}{2} \textup{arg}\big( \widehat{\psi}(R) \big) - \frac{1}{2} \textup{arg} \big( \widehat{\psi}(- R) \big) \quad \text{ modulo } \pi.
$$
We also define the complementary quantities 
$$
E_{R^c}(\psi) := E(\psi) - E_R(\psi), \quad \text{ and } \quad [P_{R^c}](\psi) = [P](\psi) - [P_R](\psi).
$$
Since $\widehat{\psi}_\infty$ is in $X(\R)$, we can find some number $R_0 >0$ such that $|\widehat{\psi}_\infty(x)| \geq 1/2$ whenever $|x| \geq R_0$. In particular, the quantity $[P_R](\psi_\infty)$ is well-defined for $R \geq R_0$. Given any integer $m \geq 1$, we choose a number $R_m \geq \max \{ R_0, m + 2 \}$ such that
\begin{equation}
\label{eq:uri}
E_{R_m^c}(\psi_\infty) \leq \frac{1}{m}, \quad \text{and} \quad \sup_{R \geq R_m}\big| [P_{R^c}](\psi_\infty) \big| \leq \frac{1}{m}.
\end{equation}
For fixed $m \geq 1$, using the convergences in~\eqref{eq:convfaibleprofil},~\eqref{eq:convfortprofil} and the local uniform convergence of the functions $\widehat{\psi}_n$, we obtain that
$$
\liminf_{n \to \infty} E_{R_m}(\psi_n) \geq E_{R_m}(\psi_\infty),
$$
and
$$
\lim_{n \to \infty} \sup_{R_m \leq R \leq 2 R_m} \big|[P_R](\psi_n) - [P_R](\psi_\infty)\big| = 0.
$$
Using a diagonal argument, we may thus extract a subsequence $(\psi_{n_m})_{m \geq 1}$ such that
\begin{equation}
\label{eq:grisons}
E_{R_m}(\psi_{n_p}) \geq E_{R_m}(\psi_\infty) - \frac{1}{m}, \quad \text{ and }\quad \sup_{R_m \leq R \leq 2 R_m} \big| [P_R](\psi_{n_p}) - [P_R](\psi_\infty) \big| \leq \frac{1}{m},
\end{equation}
for any $p \geq m$. Using a pigeonhole type argument, we next prove

\begin{lem}
\label{lem:vaud}
Given any integer $m \geq 1$, there exists a number $\widetilde{R}_m \in [R_m, 2 R_m]$ and a function $\widetilde{\psi}_{n_m} \in H^1([- \widetilde{R}_m, \widetilde{R}_m] \times \T_\ell, \C)$ such that
$$
\widetilde{\psi}_{n_m} \big( \pm \widetilde{R}_m, \cdot \big) = \psi_{n_m} \big( \pm \widetilde{R}_m, \cdot \big),
$$
$$
E_{\widetilde{R}_m} \big( \widetilde{\psi}_{n_m} \big) \leq \frac{C_\ell}{R_m}, \quad \text{ and } \quad \Big| [P_{\widetilde{R}_m}] \big( \widetilde{\psi}_{n_m} \big) \Big| \leq \frac{C_\ell}{R_m^\frac{1}{2}},
$$
for some number $C_\ell >0$, depending only on $\ell$.
\end{lem}

We then extend the function $\widetilde{\psi}_{n_m}$ to $\R \times \T_\ell$ by being equal to $\psi_{n_m}$ outside $[-\widetilde{R}_m, \widetilde{R}_m] \times \T_\ell$. We claim

\begin{prop}
\label{prop:tessin}
In the limit $n \to \infty$, we have
\begin{equation}\label{eq:tessin0}
E \big( \psi_{n_m} \big) \geq \max \big\{ E \big( \psi_\infty \big), \delta_\infty \big\} + E \big( \widetilde{\psi}_{n_m} \big) + o(1),
\end{equation}
and
\begin{equation}\label{eq:tessin1}
[P] \big( \psi_{n_m} \big) = [P] \big( \psi_\infty \big) + [P] \big( \widetilde{\psi}_{n_m} \big) + o(1),
\end{equation}
where $\delta_\infty > 0$ is the number appearing in~\eqref{eq:massezero}.
\end{prop}

We are now in position to provide the 

\begin{proof}[Completion of the proof of Theorem~\ref{thm:exist-min}]
Let $[p_\infty] := [P](\psi_\infty)$. A first consequence of Proposition~\ref{prop:tessin} is that $[p_\infty] \neq [0]$. Indeed, it would otherwise follow from~\eqref{eq:tessin1} and~\eqref{eq:tessin0} that
$$
[P] \big(\widetilde{\psi}_{n_m} \big) \to [p], \quad \text{ and } \quad \limsup_{m \to \infty} E \big( \widetilde{\psi}_{n_m} \big) \leq \boI_\textup{2d}(p) - \delta_\infty < \boI_\textup{2d}(p).
$$
This would contradict the definition of $\boI_\textup{2d}(p)$ and the fact that it depends continuously on $p$.

Assume next by contradiction that $[p_\infty] \neq [p]$. Taking limits in~\eqref{eq:tessin1} and~\eqref{eq:tessin0}, we first obtain
$$
[P] \big( \widetilde{\psi}_{n_m} \big) \to [p - p_\infty],
$$
and then
$$
\boI_\textup{2d}(p) \geq \boI_\textup{2d}(p_\infty) + \boI_\textup{2d}(p - p_\infty).
$$
This is not possible in view of the strict sub-additivity proved in Proposition~\ref{prop:prop-I-l}.

As a consequence, we have $[P](\psi_\infty) = [p]$, and we infer from~\eqref{eq:tessin0} that $\psi_\infty$ is a minimizer for $\boI_\textup{2d}(p)$. In other words, we can take $U_{\gp,\ell} = \psi_\infty$ in the statement of Theorem~\ref{thm:exist-min}. Going back to~\eqref{eq:defP2D}, and using~\cite[Lemma 4]{BeGrSaS1}, we next observe that the untwisted momentum $[P]$ is differentiable with respect to perturbations in $H^1(\R \times \T_\ell)$. So is the energy $E$, and the equation for $U_{\gp,\ell}$ in~\eqref{eq:Psi_p} is then a direct consequence of the Euler-Lagrange multiplier theorem. The smoothness of $U_{\gp,\ell}$ finally follows from standard elliptic regularity (see e.g.~\cite[Lemma 2.1]{BetGrSa1}).
\end{proof}

%%%%%%%%%%%%%%%%%%%%%%%%%%%%%
%%%%%%%%%%%%%%%%%%%%%%%%%%%%%
%%%%%%%%%%%%%%%%%%%%%%%%%%%%%
\section{The analysis of the minimizing energy}
\label{sec:energy}
%%%%%%%%%%%%%%%%%%%%%%%%%%%%%
%%%%%%%%%%%%%%%%%%%%%%%%%%%%%
%%%%%%%%%%%%%%%%%%%%%%%%%%%%%

%%%%%%%%%%%%%%%%%%%%%%%%%%%%%
%%%%%%%%%%%%%%%%%%%%%%%%%%%%%
\subsection{Proof of Lemma~\ref{lem:prop-I-l}}
%%%%%%%%%%%%%%%%%%%%%%%%%%%%%
%%%%%%%%%%%%%%%%%%%%%%%%%%%%%

Lemma~\ref{lem:prop-I-l} is exactly~\cite[Lemma 2]{deLGrSm1} up to a suitable scaling. Observe that the minimization problem under consideration in~\cite{deLGrSm1} is indeed defined as
\begin{equation}
\label{def:I-lambda}
\boJ_\lambda([q]) := \inf \big\{ E_\lambda(\psi) : \psi \in X(\R \times \T_1) \text{ s.t. } [P_1](\psi) = [q] \big\},
\end{equation}
for any number $q \in \R$. The energy $E_\lambda$ in this definition is given by
$$
E_\lambda(\psi) := \frac{1}{2} \int_{\R \times \T_1} \big( |\partial_x \psi|^2 + \lambda^2 |\partial_y \psi|^2 \big) + \frac{1}{4} \int_{\R \times \T_1} \big( 1 - |\psi|^2 \big)^2,
$$
for a number $\lambda > 0$. The untwisted momentum $[P_1]$ is exactly the untwisted momentum in~\eqref{eq:defP2D} for the specific choice $\ell = 1$. In~\cite[Lemma 2]{deLGrSm1}, we established that the function $q \mapsto \boJ_\lambda([q])$ is well-defined as a $\pi$-periodic, even and Lipschitz continuous function on $\R$, with Lipschitz constant at most $\sqrt{2}$. Moreover, it is bounded by
\begin{equation}
\label{eq:estim-sup-J}
\boJ_\lambda(q) := \boJ_\lambda([q]) \leq \boI_\textup{1d}(q) < \sqrt{2} \, q,
\end{equation}
for any $0 < q \leq \pi/2$.

In order to rephrase this statement in our current setting, we introduce the scaling $\psi_\ell(x, y) = \psi(x, \ell y) $.
When the function $\psi$ lies in $X(\R \times \T_\ell)$, the function $\psi_\ell$ belongs to $X(\R \times \T_1)$. For $\lambda = 1/\ell$, we check that
$$
E_\lambda(\psi_\ell) = E(\psi), \quad \text{ and } \quad [P_1](\psi_\ell) = [P](\psi).
$$
Combining the two previous formulae, we are led to the identity
$$
\boI_\textup{2d}(p) = \boJ_\lambda(p),
$$
for any $p \in \R$. The statements in Lemma~\ref{lem:prop-I-l} then result from the previous properties of the function $\boJ_\lambda$. \qed

%%%%%%%%%%%%%%%%%%%%%%%%%%%%%%%%
%%%%%%%%%%%%%%%%%%%%%%%%%%%%%%%%
\subsection{Proof of Lemma~\ref{lem:presymmetry}}
%%%%%%%%%%%%%%%%%%%%%%%%%%%%%%%%
%%%%%%%%%%%%%%%%%%%%%%%%%%%%%%%%

The proof combines two ingredients. The first one is an oscillation estimate.
For the ease of reference, we recall here the identity~\eqref{eq:cle} 
$$
[P] \big( \psi_n(\cdot, y) \big) = [P] \big( \widehat{\psi}_n \big) + \frac{1}{2} \int_\R \big\langle i \partial_x w_n(x, y), w_n(x, y) \big\rangle_\C \, dx,
$$
and we denote by $q_n$ the real-valued function defined by
$$
q_n(y) = \frac{1}{2} \int_\R \langle i \partial_x w_n(x, y), w_n(x, y) \rangle_\C \, dx,
$$
for any $y \in \T_\ell$. Since $w_n \in H^1(\R \times \T_\ell),$ the function $q_n$ belongs to $W^{1, 1}(\T_\ell)$, with 
\begin{equation}
\label{eq:cont-deriv-q}
q_n'(y) = \int_\R \big\langle i \partial_x w_n(x, y), \partial_y w_n(x, y) \big\rangle_\C \, dx.
\end{equation}
Moreover, we compute
\begin{equation}
\label{eq:borne-deriv-q}
\int_{\T_\ell} |q_n'(y)| \, dy \leq \frac{1}{2} \int_{\R \times \T_\ell} \Big( |\partial_x w_n|^2 + |\partial_y w_n|^2 \Big)
\leq \ell E(\psi_n) \leq C,
\end{equation}
where $C$ is a positive number depending only on $\ell$ and $p$. 

The second ingredient is Lemma~\ref{lem:osci} below, which we applied to the functions $q_n - \int_{\T_\ell} q_n/\ell$ for $n$ large enough, and next combined with equality~\eqref{eq:cle}. Indeed, consider the integer $N_0 \geq 1$ and the number $\delta_0 > 0$ given by Lemma~\ref{lem:osci} for $\sigma = \ell (\boI_\textup{1d}(p) - \boI_\textup{2d}(p))/2$, and for the number $C$ in~\eqref{eq:borne-deriv-q}. In view of~\eqref{eq:defsigma} and~\eqref{eq:borne-deriv-q}, the functions $q_n - \int_{\T_\ell} q_n/\ell$ satisfy the assumptions of Lemma~\ref{lem:osci} for any integer $n$ large enough. Whenever $N = 2^j \geq N_0$ and $|\delta| < \delta_0$, we can therefore find an interval $I_n$ of size $\ell/2^j$ for which 
$$
\frac{1}{|I_n|} \int_{I_n} \bigg( q_n - \frac{1}{\ell} \int_{\T_\ell} q_n \bigg) = \delta.
$$
In view of~\eqref{eq:cle}, this can be rewritten as 
$$
\frac{1}{|I_n|} \int_{I_n} [P] \big( \psi_n(\cdot,y) \big) \, dy = [P] \big( \widehat{\psi}_n \big) + \frac{1}{\ell}\int_{\T_\ell} q_n(y) \, dy + \delta = [P] \big( \psi_n \big) + \delta = [p + \delta],
$$
and the proof of Lemma~\ref{lem:presymmetry} is therefore completed once we have proved

\begin{lem}
\label{lem:osci}
Let $C, \sigma >0$ and let $q \in W^{1,1}(\T_\ell, \R)$ be such that 
$$
\int_{\T_\ell} q = 0, \quad \int_{\T_\ell} |q'| \leq C, \quad \text{ and } \quad
	\int_{\T_\ell} |q| \geq \sigma.
$$
There exist an integer $N_0$ and a number $\delta_0 > 0$, depending only on $C$, $\sigma$ and $\ell$, such that, given any integer $N \geq N_0$ and any number $0 \leq |\delta| \leq \delta_0$, there exists at least one subinterval $I$ of $\T_\ell$ such that
$$
|I| = \frac{\ell}{N}, \quad \text{and} \quad \frac{1}{|I|} \int_I q = \delta.
$$
\end{lem}

\begin{proof}
Decompose the function $q$ as $q = q_+ - q_-$, where $q_+$ is the non-negative part of $q$, and $q_-$ the opposite of its non-positive part. Since $|q| = q_+ + q_-$, we deduce from the integral assumptions in Lemma~\ref{lem:osci} that
$$
\int_{\T_\ell} q_+ = \int_{\T_\ell} q_- \geq \frac{\sigma}{2}.
$$
Consider next the intervals $I_k = [k \ell/N, (k + 1) \ell /N]$ for $0 \leq k \leq N - 1$, with length $h = \ell/N$. Fix a number $\delta > 0$, and set
$$
\boJ_\pm = \Big\{ 0 \leq k \leq N - 1 \text{ s.t. } \frac{1}{h} \int_ {I_k} q_\pm > 0 \Big\}, \quad \text{ and } \quad \boK_\pm = \Big\{ 0 \leq k \leq N - 1 \text{ s.t. } \frac{1}{h} \int_ {I_k} q_\pm \geq \delta \Big\}.
$$
When an integer $k$ lies in $\boJ_- \cap \boK_+$, we deduce from the continuity of the function $q$ the existence of a number $\xi \in I_k$ such that $q(\xi) = 0$. As a consequence, we have
$$
\delta \leq \frac{1}{h} \int_{I_k} \big| q_+(x) - q_+(\xi) \big| \, dx \leq \frac{1}{h} \int_{I_k} \bigg| \int_\xi^x q_+'(y) \, dy \bigg| \leq \int_{I_k} |q_+'| \leq \int_{I_k} |q'|,
$$
and $\boJ_- \cap \boK^+$ is a subset of
$$
\boL = \Big\{ 0 \leq k \leq N - 1 \text{ s.t. } \int_ {I_k} \big| q' \big| \geq \delta \Big\}.
$$
Similarly, we show that $\boJ_+ \cap \boK_-$ is a subset of $\boL$. Note here that the cardinal of $\boL$ is less than $C/\delta$ due to the Tchebychev inequality, so that the same bound holds for both the sets $\boJ_- \cap \boK_+$ and $\boJ_+ \cap \boK_-$.

Since the integral $\int_{\T_\ell} q$ vanishes, there also exists a number $y \in \T_\ell$ such that $q(y) = 0$. As a consequence, we have
$$
q_\pm(x) = q_\pm(x) - q_\pm(y) \leq \int_{T_\ell} |q_\pm'| \leq \int_{T_\ell} |q'| \leq C,
$$
for any $x \in \T_\ell$, so that
$$
\int_{I_k} q_\pm \leq C h,
$$
for any $0 \leq k \leq N - 1$. Since $\int_{\T_\ell} q_+ \geq \sigma/2$, we deduce that
$$
\frac{\sigma}{2} \leq \sum_{I \in \boK_+} \int_I q_+ + \sum_{I \notin \boK_+} \int_I q_+ \leq C h \, \text{Card} \big( \boK_+ \big) + \delta h \, \text{Card} \big( \boK_+^c \big).
$$
Observing that $\text{Card}(\boK_+^c) = N - \text{Card}(\boK_+)$, we are led to
$$
\text{Card}(\boK_+) \geq \frac{N (\sigma - 2 \delta \ell)}{2 (C - \delta) \ell},
$$
when $\delta < C$, and the same estimate holds for the cardinal of $\boK_-$.

In conclusion, we obtain
$$
\text{Card} \big( \boK_+ \cap \boJ_-^c \big) \geq \frac{N (\sigma - 2 \delta \ell)}{2 (C - \delta) \ell} - \frac{C}{\delta},
$$
as well as the same estimate for the cardinal of $\boK_- \cap \boJ_+^c$. We finally fix $\delta = \delta_0 = \min \{ \sigma/(4 \ell), \linebreak[0] C/2 \}$, so that we can find a number $N_0 > 0$, depending only on $C$, $\sigma$ and $\ell$, such that
$$
\text{Card} \big( \boK_\pm \cap \boJ_\mp^c \big) \geq 1,
$$
for any $N \geq N_0$. Choosing integers $k_\pm \in \boK_\pm \cap \boJ_\mp^c$, we are led to
$$
\frac{1}{h} \int_{I_{k_+}} q = \frac{1}{h} \int_{I_{k_+}} q_+ \geq \delta_0,
	\quad \text{ and } \quad \frac{1}{h} \int_{I_{k_-}} q = - \frac{1}{h}
	\int_{I_{k_-}} q_- \leq - \delta_0.
$$

We complete the proof by recalling that $q$ is in $L^1(\T_\ell)$, so that the map $t \mapsto \int_t^{t + h} q$ is continuous on $\T_\ell$. Given any number $|\delta| \leq \delta_0$, it is then sufficient to apply the intermediate value theorem to find an interval $I_\delta$, with length $h = \ell/N$, such that
$$
\frac{1}{h} \int_ {I_\delta} q = \delta.
$$
This completes the proof of Lemma~\ref{lem:osci}.
\end{proof}

%%%%%%%%%%%%%%%%%%%%%%%%%%%%%
%%%%%%%%%%%%%%%%%%%%%%%%%%%%%
\subsection{Proof of Lemma~\ref{lem:deriv-0}}
%%%%%%%%%%%%%%%%%%%%%%%%%%%%%
%%%%%%%%%%%%%%%%%%%%%%%%%%%%%

For the proof of Lemma~\ref{lem:deriv-0}, we shall make use of the notion of vanishing sequence introduced in Subsection~\ref{sub:conc-comp} and the proofs of the related Lemmas~\ref{lem:conv-const},~\ref{lem:lions} and~\ref{lem:farfromone} (which could better be read first). 
Set
$$
\tau = \liminf_{p \to 0^+} \frac{\boI_\textup{2d}(p)}{p}.
$$
In view of~\eqref{eq:estim-sup-I}, we already know that
$$
\limsup_{p \to 0^+} \frac{\boI_\textup{2d}(p)}{p} \leq \sqrt{2},
$$
so that $\tau \leq \sqrt{2}$, and the convergence in Lemma~\ref{lem:deriv-0} will follow from the inequality
\begin{equation}
\label{asm}
\tau \geq \sqrt{2}.
\end{equation}

In order to prove this inequality, we consider a sequence of positive numbers $(p_n)_{n \geq 0}$, such that $p_n \to 0$ as $n \to \infty$, and
\begin{equation}
\label{haouas}
\frac{\boI_\textup{2d}(p_n)}{p_n} \to \tau.
\end{equation}
By definition of the minimizing energy $\boI_\textup{2d}(p_n)$, given any number $0 < \delta \leq 2 - \sqrt{2}$, there exist functions $\psi_n \in X(\R \times \T_\ell)$, with $[P_\ell](\psi_n) = p_n$ modulo $\pi$, such that
\begin{equation}
\label{kremer}
\boI_\textup{2d}(p_n) \leq E(\psi_n) \leq \boI_\textup{2d}(p_n) + \delta p_n \leq \big( \tau + \delta \big) p_n \leq 2 p_n,
\end{equation}
for any $n \geq 0$. Since $p_n \to 0$ as $n \to \infty$, it follows in particular that $(\psi_n)_{n \geq 0}$ is vanishing, and therefore also from Lemma~\ref{lem:conv-const} that
\begin{equation}
\label{sowakula}
\delta_n = \inf_{x \in \R} \big| \widehat{\psi}_n(x) \big| \to 1, 
\end{equation}
as $n \to \infty$. 
By~\eqref{kremer}, the energies $E(\psi_n)$ are also uniformly bounded and Lemma~\ref{lem:farfromone} provides a subsequence, still denoted $(\psi_n)_{n \geq 0}$ here, for which
$$
p_n \leq \frac{E(\psi_n)}{\sqrt{2} \, \delta_n} + o(1),
$$
as $n \to \infty$. In view of~\eqref{haouas},~\eqref{kremer} and~\eqref{sowakula}, proving~\eqref{asm} would follow from the fact that the $o(1)$ in this inequality is actually a $o(p_n)$. In order to derive this stronger estimate, we adapt slightly the argument in the proof of Lemma~\ref{lem:farfromone}.

Going first to the proof of Lemma~\ref{lem:lions}, we decompose the energies $E(\psi_n)$ as $E(\psi_n) = \boE_n + \boR_n$ according to~\eqref{unsa}. Arguing as in the proof of~\eqref{sud} (for instance with $r = 1$), we obtain
\begin{equation}
\label{urios}
\int_\R \int_{\T_\ell} |w_n|^4 \leq C_\ell E(\psi_n)^2,
\end{equation}
for any $n \geq 0$. Here as in the sequel, the number $C_\ell \geq 0$, possibly different from line to line, only depends on $\ell$. Inserting the previous inequality into~\eqref{cgc}, and applying~\eqref{kremer}, we obtain
\begin{equation}
\label{faingaa}
\big| \boR_n \big| \leq C_\ell E(\psi_n) \Big( {\boE_n}^\frac{1}{2} + E(\psi_n) \Big) \leq 2 C_\ell \, p_n \big( {\boE_n}^\frac{1}{2} + 2 p_n \big).
\end{equation}
Since $\boE_n = E(\psi_n) - \boR_n$, we deduce from~\eqref{kremer} and~\eqref{faingaa} that
\begin{equation}
\label{charrier}
\boE_n \leq C_\ell p_n,
\end{equation}
and then from~\eqref{faingaa} that
$$
\big| E(\psi_n) - \boE_n \big| \leq C_\ell p_n^\frac{3}{2} \big( 1 + p_n^\frac{1}{2} \big).
$$
In view of~\eqref{kremer}, we conclude that
\begin{equation}
\label{falgoux}
\boE_n \leq \big( \tau + \delta \big) p_n + C_\ell p_n^\frac{3}{2} \big( 1 + p_n^\frac{1}{2} \big).
\end{equation}

Going back to~\eqref{sowakula}, we can assume, up to a subsequence, that $\delta_n \geq 1/2$ for any $n \geq 0$. In this case, the quantity $\gR_n$ in~\eqref{cftc} is well-defined, and we can bound it by
\begin{equation}
\label{lee}
\big| \gR_n \big| \leq C_\ell \, \big\| w_n \big\|_{L^4(\R \times \T_\ell}^2 \, \big\| \partial_x \widehat{\psi}_n \|_{L^2(\R \times \T_\ell)} \leq C_\ell \, E(\psi_n) \, \boE_n \leq C_\ell \, p_n^2,
\end{equation}
by combining~\eqref{kremer},~\eqref{urios} and~\eqref{charrier}.

Arguing as in the proof of Lemma~\ref{lem:farfromone}, we can also lift the function $\widehat{\psi}_n$ as $\widehat{\psi}_n = \rho_n \, e^{i \varphi_n}$, and find an integer $k_n \in \Z$ such that
$$
p_n + k_n \pi = \frac{1}{2} \int_\R \big( 1 - \rho_n^2 \big) \partial_x \varphi_n + \frac{1}{\ell} \int_\R \int_{\T_\ell} \frac{1}{|\widehat{\psi}_n|^2} \, \langle w_n, \widehat{\psi}_n \rangle_\C \, \langle i \partial_x w_n , \widehat{\psi}_n \rangle_\C + \boR_n.
$$
Since $p_n \to 0$ as $n \to \infty$, we infer from the inequality $|\widehat{\psi}_n| = \rho_n \geq \delta_n$, and from~\eqref{lee} that
$$
p_n \leq \big| p_n + k_n \pi \big| \leq \frac{1}{2 \delta_n} \int_\R \big| 1 - \rho_n^2 \big| \big| \rho_n \partial_x \varphi_n \big| + \frac{1}{\ell \delta_n} \int_\R \int_{\T_\ell} \big| \langle w_n, \widehat{\psi}_n \rangle_\C \big| \, \big| \partial_x w_n \big| + C_\ell \, p_n^2.
$$

On the other hand, we also check that
$$
\boE_n = \frac{1}{2} \int_\R \Big( |\nabla \rho_n|^2 + \rho_n^2 |\nabla \varphi_n|^2 + \frac{(1 - \rho_n^2)^2}{2} \Big) + \frac{1}{2 \ell} \int_\R \int_{\T_\ell} \Big( |\nabla w_n|^2 + 2 \langle w_n, \widehat{\psi}_n \rangle_\C^2 \Big).
$$
As in the proof of Lemma~\ref{lem:farfromone}, we derive from the Young inequality, and then from~\eqref{falgoux}, that 
$$
p_n \leq \frac{\boE_n}{\sqrt{2} \, \delta_n} + C_\ell \, p_n^2 \leq \frac{\tau + \delta}{\sqrt{2} \, \delta_n} p_n + C_\ell p_n^\frac{3}{2} \big( 1 + p_n^\frac{1}{2} \big).
$$
Using the fact that $p_n \to 0$ and $\delta_n \to 1$ as $n \to \infty$, we obtain in this limit that $
\tau + \delta \geq \sqrt{2}.$ Since $\delta$ can be chosen as any arbitrary small positive number, the inequality in~\eqref{asm} is proved, and as a consequence, so is Lemma~\ref{lem:deriv-0}. \qed

%%%%%%%%%%%%%%%%%%%%%%%%%%%%%%%%
%%%%%%%%%%%%%%%%%%%%%%%%%%%%%%%%
\subsection{Proof of Proposition~\ref{prop:prop-I-l}}
%%%%%%%%%%%%%%%%%%%%%%%%%%%%%%%%
%%%%%%%%%%%%%%%%%%%%%%%%%%%%%%%%

The concavity of the function $\boI_\textup{2d}$ on the interval $[0, \pi]$ is a direct consequence of Lemmas~\ref{lem:prop-I-l},~\ref{lem:concave} and~\ref{lem:cond-concave}. Combining Lemmas~\ref{lem:strict-subadditivity} and~\ref{lem:deriv-0} with the strict inequality in~\eqref{eq:estim-sup-I} guarantees that the function $\boI_\textup{2d}$ is also strictly sub-additive on $[0, \pi]$. We now use its $\pi$-periodicity and its parity to extend this property to $\R / \pi \Z$.

Fix two numbers $(\gp_1, \gp_2) \in (\R / \pi \Z)^2$, with $\gp_1 \neq 0$ and $\gp_2 \neq 0$, and consider two numbers $0 < p_1, p_2 < \pi$ such that $\gp_1 = [p_1]$ and $\gp_2 = [p_2]$. Two situations can occur. When $p_1 + p_2 \leq \pi$, the strict sub-additivity of the function $\boI_\textup{2d}$ on $[0, \pi]$ directly provides
$$
\boI_\textup{2d}(\gp_1 + \gp_2) = \boI_\textup{2d}(p_1 + p_2) < \boI_\textup{2d}(p_1) + \boI_\textup{2d}(p_2) = \boI_\textup{2d}(\gp_1) + \boI_\textup{2d}(\gp_2).
$$
Otherwise, we know that $\pi < p_1 + p_2 < 2 \pi$, so that $0 < \pi - p_1 + \pi - p_2 < \pi$. Since $0 < \pi - p_1 < \pi$ and $0 < \pi - p_2 < \pi$, we can again invoke the strict sub-additivity of the function $\boI_\textup{2d}$ on $[0, \pi]$ to obtain
$$
\boI_\textup{2d}(\pi - p_1 + \pi - p_2) < \boI_\textup{2d}(\pi - p_1) + \boI_\textup{2d}(\pi - p_2).
$$
Since the function $\boI_\textup{2d}$ is even and $\pi$-periodic, this can be rewritten as
$$
\boI_\textup{2d}(p_1 + p_2) < \boI_\textup{2d}(p_1) + \boI_\textup{2d}(p_2),
$$
which again leads to
$$
\boI_\textup{2d}(\gp_1 + \gp_2) < \boI_\textup{2d}(\gp_1) + \boI_\textup{2d}(\gp_2).
$$
This concludes the proofs of~\eqref{eq:str-sub-add} and of Proposition~\ref{prop:prop-I-l}. \qed

%%%%%%%%%%%%%%%%%%%%%%%%%%%%%%%%%%
%%%%%%%%%%%%%%%%%%%%%%%%%%%%%%%%%%
%%%%%%%%%%%%%%%%%%%%%%%%%%%%%%%%%%
\section{The compactness of the minimizing sequences}
\label{sec:compact}
%%%%%%%%%%%%%%%%%%%%%%%%%%%%%%%%%%
%%%%%%%%%%%%%%%%%%%%%%%%%%%%%%%%%%
%%%%%%%%%%%%%%%%%%%%%%%%%%%%%%%%%%

%%%%%%%%%%%%%%%%%%%%%%%%%%%%%%%
%%%%%%%%%%%%%%%%%%%%%%%%%%%%%%%
\subsection{Proof of Lemma~\ref{lem:conv-const}}
%%%%%%%%%%%%%%%%%%%%%%%%%%%%%%%
%%%%%%%%%%%%%%%%%%%%%%%%%%%%%%%

Assume for the sake of a contradiction that 
$$
\liminf_{n \to \infty} \, \sup_{x \in \R} \, \big| 1 - |\widehat{\psi}_n(x)| \big| = \nu > 0.
$$
In this case, we can find a sequence of points $(a_n)_{n \geq 0}$ in $\R$ such that
\begin{equation}
\label{cfdt}
\big| \widehat{\psi}_n(a_n) \big| \leq 1 - \frac{\nu}{2},
\end{equation}
for $n$ large enough. Given a compact subset $K$ of $\R$, which contains $0$, we deduce from
condition~\eqref{eq:vanishing} that
$$
\int_K \int_{\T_\ell} \Big( |\nabla \psi_n(\cdot + a_n, \cdot)|^2 + (1 - |\psi_n(\cdot + a_n, \cdot)|^2)^2 \Big) \to 0, 
$$
as $n \to \infty$. As a consequence, we obtain
\begin{equation}
\label{cgt}
\nabla \psi_n(\cdot + a_n, \cdot) \to 0 \quad \text{ in } L^2(K \times \T_\ell), \quad \text{ and } \quad
1 - |\psi_n(\cdot + a_n, \cdot)|^2 \to 0 \quad \text{ in } L^2(K \times \T_\ell).
\end{equation}
Moreover, we also have
$$
\int_K \int_{\T_\ell} |\psi_n(\cdot + a_n, \cdot)|^2 \leq \ell |K| + \big( \ell |K| \big)^\frac{1}{2}
\bigg( \int_K \int_{\T_\ell} \big( 1 - |\psi_n(\cdot + a_n, \cdot)|^2 \big)^2 \bigg)^\frac{1}{2},
$$
for any $n \geq 0$. Up to a subsequence, we can therefore assume that there exists a function
$\psi_\infty \in H^1(K \times \T_\ell)$ such that
$$
\psi_n(\cdot + a_n, \cdot) \rightharpoonup \psi_\infty \quad \text{ in } H^1(K \times \T_\ell).
$$
Combining~\eqref{cgt} with the Rellich-Kondrachov theorem, we check that the function $\psi_\infty$
is a constant function of modulus $1$, and that the previous convergence is actually strong in
$H^1(K \times \T_\ell)$.

Observe finally that
$$
\int_K |\widehat{\psi}_n(\cdot + a_n) - \psi_\infty|^2 \leq \frac{1}{\ell} \int_K \int_{\T_\ell} |\psi_n(\cdot + a_n,
\cdot) - \psi_\infty|^2 \to 0,
$$
while
$$
\int_K |\partial_x \widehat{\psi}_n(\cdot + a_n) |^2 \leq \frac{1}{\ell} \int_K \int_{\T_\ell} |\partial_x
\psi_n(\cdot + a_n) |^2 \to 0.
$$
Hence we have
$$
\widehat{\psi}_n(\cdot + a_n) \to \psi_\infty \quad \text{ in } H^1(K),
$$
as $n \to \infty$. Invoking the Rellich compactness theorem, we can assume, up to a further subsequence, that
$$
\widehat{\psi}_n(\cdot + a_n) \to \psi_\infty \quad \text{ in } L^\infty(K).
$$
Since $0 \in K$, this is enough to guarantee that
$$
|\widehat{\psi}_n(a_n)| \to |\psi_\infty| = 1,
$$
which provides a contradiction with~\eqref{cfdt}. This concludes the proof of Lemma~\ref{lem:conv-const}. \qed

%%%%%%%%%%%%%%%%%%%%%%%%%%%
%%%%%%%%%%%%%%%%%%%%%%%%%%%
\subsection{Proof of Lemma~\ref{lem:lions}}
%%%%%%%%%%%%%%%%%%%%%%%%%%%
%%%%%%%%%%%%%%%%%%%%%%%%%%%

Concerning the proof of~\eqref{eq:retraite}, we set $\psi_n = \widehat{\psi}_n + w_n$ as
in~\eqref{eq:decomp1d}. Going back to~\eqref{eq:decompEw}, we decompose the energy $E(\psi_n)$ as
\begin{equation}
\label{unsa}
E( \psi_n) = \boE_n + \boR_n,
\end{equation}
with
$$
\boE_n := \frac{1}{2 \ell} \int_\R \int_{\T_\ell} \Big( |\partial_x \widehat{\psi}_n|^2 + \frac{1}{2} \big( 1 - |\widehat{\psi}_n|^2 \big)^2 + |\nabla w_n|^2 + 2 \langle w_n, \widehat{\psi}_n \rangle_\C^2 \Big),
$$
and
$$
\boR_n := \frac{1}{2 \ell} \int_\R \int_{\T_\ell} \Big( - |w_n|^2 \big( 1 - |\widehat{\psi}_n|^2 \big) + 2 \langle w_n, \widehat{\psi}_n \rangle_\C \, |w_n|^2 + \frac{1}{2} |w_n|^4 \Big).
$$
We have to check that the term $\boR_n$ tends to $0$ as $n \to \infty$. Using the Cauchy-Schwarz inequality, we observe that
\begin{equation}
\label{cgc}
\begin{split}
\big| \boR_n \big| \leq & \frac{1}{2 \ell} \big\| w_n \big\|_{L^4(\R \times \T_\ell)}^2 \Big( \big\| 1 - |\widehat{\psi}_n|^2 \big\|_{L^2(\R \times \T_\ell)}+ 2 \big\| \langle w_n, \widehat{\psi}_n \rangle_\C \big\|_{L^2(\R \times \T_\ell)} + \frac{1}{2} \big\| w_n \big\|_{L^4(\R \times \T_\ell)}^2 \Big) \\
\leq & \frac{1}{2 \ell} \big\| w_n \big\|_{L^4(\R \times \T_\ell)}^2 \Big( 4 \, \ell^\frac{1}{2} \, {\boE_n}^\frac{1}{2} + \big\| w_n \big\|_{L^4(\R \times \T_\ell)}^2 \Big).
\end{split}
\end{equation}
As a consequence, we are reduced to prove that, under condition~\eqref{eq:vanishing},
\begin{equation}
\label{fo}
\big\| w_n \big\|_{L^4(\R \times \T_\ell)} \to 0,
\end{equation}
as $n \to \infty$. The proof of this claim is classical in the context of concentration-compactness arguments
(see~\cite[Lemma I.1]{Lions2}). For the sake of completeness, we give the following detail.

Fix the positive number $r$ such that~\eqref{eq:vanishing} holds, and set
$$
\varepsilon_n := \sup_{a \in \R} \, \frac{1}{\ell} \int_{B(a,r)} \int_{\T_\ell} e(\psi_n).
$$
Invoking the Poincar\'e-Wirtinger inequality, we check that
$$
\big\| w_n \big\|_{L^2(B(a, r) \times \T_\ell)} \leq \big\| \partial_y w_n \big\|_{L^2(B(a, r) \times \T_\ell)} \leq \big\| \nabla w_n \big\|_{L^2(B(a, r) \times \T_\ell)},
$$
for any $n \geq 0$ and any $a \in \R$. Combining this inequality with the Gagliardo-Nirenberg inequality, we can find a number $C_r$, only depending on $r$, such that
\begin{align*}
\int_{B(a, r)} \int_{\T_\ell} |w_n|^4 \leq & C_r \bigg( \int_{B(a, r)} \int_{\T_\ell} |w_n|^2 \bigg) \bigg( \int_{B(a, r)} \int_{\T_\ell} \big( |\nabla w_n|^2 + |w_n|^2 \big)\bigg) \\
\leq & 2 \, C_r \bigg( \int_{B(a, r)} \int_{\T_\ell} |\nabla w_n|^2 \bigg)^2.
\end{align*}
Recall next that
$$
\int_{B(a, r)} \int_{\T_\ell} |\nabla w_n|^2 \leq 2 \int_{B(a,r)} \int_{\T_\ell} e(\psi_n) \leq 2 \, \ell \, \varepsilon_n,
$$
so that
$$
\int_{B(a, r)} \int_{\T_\ell} |w_n|^4 \leq 8 \, \ell \, C_r \, \varepsilon_n \, \int_{B(a, r)} \int_{\T_\ell} e(\psi_n).
$$
By summation we are then led to
\begin{equation}
\label{sud}
\int_\R \int_{\T_\ell} |w_n|^4 \leq 8 \, \ell \, C_r \, \varepsilon_n \, E(\psi_n).
\end{equation}
Since $\varepsilon_n \to 0$ as $n \to \infty$ by~\eqref{eq:vanishing}, and since the energies $E(\psi_n)$ are bounded, we conclude that~\eqref{fo} does hold. Inserting this limit into~\eqref{unsa} and~\eqref{cgc}, and using again the boundedness of the energies $E(\psi_n)$, we deduce that the sequence $(\boE_n)_{n \geq 0}$ is also bounded, and then that $\boR_n$ converges to $0$ as $n \to \infty$. In view of~\eqref{unsa}, this completes the proof of~\eqref{eq:retraite}.

We now turn to~\eqref{eq:tauxplein}. In view of Lemma~\ref{lem:conv-const}, we can assume, up to a subsequence, that all the functions $\widehat{\psi}_n$ satisfy
$$
\delta_n := \inf_{x \in \R} \big| \widehat{\psi}_n(x) \big| \geq \delta_*,
$$
for a fixed number $\delta_* > 0$. In this case, we can write
$$
w_n = \langle w_n, \widehat{\psi}_n \rangle_{\C} \, \frac{\widehat{\psi}_n}{|\widehat{\psi}_n|^2} + \langle w_n, i \widehat{\psi}_n \rangle_{\C} \, \frac{i \, \widehat{\psi}_n}{|\widehat{\psi}_n|^2},
$$
so that
$$
\frac{1}{2 \ell} \int_\R \int_{\T_\ell} \langle i \partial_x w_n, w_n \rangle_\C = \frac{1}{2 \ell} \int_\R \int_{\T_\ell}
\frac{1}{|\widehat{\psi}_n|^2} \Big( \langle w_n, \widehat{\psi}_n \rangle_\C \, \langle i \partial_x w_n , \widehat{\psi}_n
\rangle_\C + \langle w_n, i \widehat{\psi}_n \rangle_\C \, \langle \partial_x w_n , \widehat{\psi}_n
\rangle_\C \Big).
$$
Integrating by parts the last term in this identity, we obtain
\begin{equation}
\label{cftc}
\frac{1}{2 \ell} \int_\R \int_{\T_\ell} \langle i \partial_x w_n, w_n \rangle_\C = \frac{1}{\ell} \int_\R \int_{\T_\ell}
\frac{1}{|\widehat{\psi}_n|^2} \langle w_n, \widehat{\psi}_n \rangle_\C \, \langle i \partial_x w_n , \widehat{\psi}_n
\rangle_\C + \gR_n,
\end{equation}
with
\begin{align*}
\gR_n = \frac{1}{2 \ell} \int_\R \int_{\T_\ell} & \frac{1}{|\widehat{\psi}_n|^2} \Big( \frac{2}{|\widehat{\psi}_n|^2} \langle \widehat{\psi}_n, \partial_x \widehat{\psi}_n \rangle_\C \, \langle w_n , \widehat{\psi}_n \rangle_\C \, \langle w_n , i \widehat{\psi}_n \rangle_\C \\
& - \langle w_n , \partial_x \widehat{\psi}_n \rangle_\C \, \langle w_n , i \widehat{\psi}_n \rangle_\C - \langle w_n , \widehat{\psi}_n \rangle_\C \, \langle w_n , i \partial_x \widehat{\psi}_n \rangle_\C \Big).
\end{align*}
Since
$$
\big| \gR_n \big| \leq \frac{2}{\ell \, \delta_*} \int_\R \int_{\T_\ell} |w_n|^2 \, |\partial_x \widehat{\psi}_n| \leq \frac{2}{\ell \, \delta_*} \, \big\| w_n \big\|_{L^4(\R \times \T_\ell)}^2 \, \big\| \partial_x \widehat{\psi}_n \big\|_{L^2(\R \times \T_\ell)},
$$
we infer from~\eqref{eq:retraite} and~\eqref{fo} that $\gR_n \to 0$ as $n \to \infty$. In view of~\eqref{cftc}, this concludes the proof of~\eqref{eq:tauxplein}, and of Lemma~\ref{lem:lions}. \qed

%%%%%%%%%%%%%%%%%%%%%%%%%%%%%%%
%%%%%%%%%%%%%%%%%%%%%%%%%%%%%%%
\subsection{Proof of Lemma~\ref{lem:farfromone}}
%%%%%%%%%%%%%%%%%%%%%%%%%%%%%%%
%%%%%%%%%%%%%%%%%%%%%%%%%%%%%%%

Recall first that
$$
\delta_n \geq 1 - \sup_{x \in \R} \big| 1 - |\widehat{\psi}_n(x)| \big|,
$$
so that by Lemma~\ref{lem:conv-const}, we can assume, up to a subsequence, that
$$
\delta_n > 0,
$$
for any $n \geq 0$. In particular, we can lift the function $\widehat{\psi}_n$ as $\widehat{\psi}_n = \rho_n e^{i \varphi_n}$, where $\rho_n$ and $\varphi_n$ are continuous real-valued functions defined on the whole line $\R$. In this situation, it follows from~\cite[Lemma 3]{BeGrSaS1} that the untwisted momentum $[P](\widehat{\psi}_n)$ is equal to
$$
[P] \big( \widehat{\psi}_n \big) = \frac{1}{2} \int_\R \big( 1 - \rho_n^2 \big) \partial_x \varphi_n \quad \text{ modulo } \pi.
$$
Going back to~\eqref{eq:defP2D} and~\eqref{eq:tauxplein}, we obtain
\begin{equation}
\label{eq:decompPfin}
[P] \big( \psi_n \big)
%\gq_n + k_n \pi
= \frac{1}{2} \int_\R \big( 1 - \rho_n^2 \big) \partial_x \varphi_n + \frac{1}{\ell} \int_\R \int_{\T_\ell} \frac{1}{|\widehat{\psi}_n|^2} \, \langle w_n, \widehat{\psi}_n \rangle_\C \, \langle i \partial_x w_n , \widehat{\psi}_n \rangle_\C + o(1) \quad \text{ modulo } \pi,
\end{equation}
as $n \to \infty$.

Similarly, we can expand the energy of the function $\widehat{\psi}_n$ as
$$
E \big( \widehat{\psi}_n \big) = \frac{1}{2} \int_\R \Big( |\nabla \rho_n|^2 + \rho_n^2 |\nabla \varphi_n|^2 + \frac{(1 - \rho_n^2)^2}{2} \Big),
$$
so that we derive from~\eqref{eq:retraite} that
\begin{equation}
\label{eq:decompEfin}
E(\psi_n) = \frac{1}{2} \int_\R \Big( |\nabla \rho_n|^2 + \rho_n^2 |\nabla \varphi_n|^2 + \frac{(1 - \rho_n^2)^2}{2} \Big) + \frac{1}{2 \ell} \int_\R \int_{\T_\ell} \Big( |\nabla w_n|^2 + 2 \langle w_n, \widehat{\psi}_n \rangle_\C^2 \Big) + o(1).
\end{equation}

We next estimate the first integral term in the right-hand side of~\eqref{eq:decompPfin} as
$$
\sqrt{2} \delta_n \bigg| \int_\R \big( 1 - \rho_n^2 \big) \partial_x \varphi_n \bigg| \leq 2 \int_\R \frac{\big| 1 - \rho_n^2 \big|}{\sqrt{2}} \, \rho_n \big| \partial_x \varphi_n \big| \leq \int_\R \Big( \rho_n^2 |\nabla \varphi_n|^2 + \frac{(1 - \rho_n^2)^2}{2} \Big),
$$
while we bound the second one by
\begin{align*}
\sqrt{2} \delta_n \bigg| \int_\R \int_{\T_\ell} \frac{1}{|\widehat{\psi}_n|^2} \, \langle w_n, \widehat{\psi}_n \rangle_\C \, \langle i \partial_x w_n , \widehat{\psi}_n \rangle_\C \bigg| \leq & \int_\R \int_{\T_\ell} \sqrt{2} \big| \langle w_n, \widehat{\psi}_n \rangle_\C \big| \, \big| \partial_x w_n \big| \\
\leq & \frac{1}{2} \int_\R \int_{\T_\ell} \Big( |\partial_x w_n|^2 + 2 \langle w_n, \widehat{\psi}_n \rangle_\C^2 \Big).
\end{align*}
In view of~\eqref{eq:decompPfin} and~\eqref{eq:decompEfin}, we are led to
$$
\big| [P](\psi_n) \big| \leq \frac{E(\psi_n) + o(1)}{\sqrt{2} \delta_n} + o(1),
$$
as $n \to \infty$. Using the fact that $\delta_n \to 1$ in this limit by Lemma~\ref{lem:conv-const}, we obtain~\eqref{eq:contre-64}, which concludes the proof of Lemma~\ref{lem:farfromone}. \qed

%%%%%%%%%%%%%%%%%%%%%%%%%%%
%%%%%%%%%%%%%%%%%%%%%%%%%%%
\subsection{Proof of Lemma~\ref{lem:vaud}}
%%%%%%%%%%%%%%%%%%%%%%%%%%%
%%%%%%%%%%%%%%%%%%%%%%%%%%%

We split the proof into four steps.

\begin{step}
\label{S1}
There exists a number $M_p > 0$, depending only on $p$, such that, given any integer $m \geq 1$, there exists a number $R_m \leq \widetilde{R}_m \leq 2 R_m$ such that
\begin{equation}
\label{eq:tranche}
\frac{1}{\ell} \int_{\T_\ell} \Big( e \big( \psi_{n_m} \big)(\widetilde{R}_m, y) + e \big( \psi_{n_m} \big)(-\widetilde{R}_m, y) \Big) \, dy \leq \frac{M_p}{R_m}.
\end{equation}
\end{step}

This follows from the fact that $(\psi_{n_m})_{m \geq 1}$ is a minimizing sequence for the minimization problem $\boI_\textup{2d}(p)$. As a consequence, we can find a positive number $M_p$, depending only on $p$, such that
$$
\frac{1}{\ell} \int_{R_m}^{2 R_m} \int_{\T_\ell} \Big( e \big( \psi_{n_m} \big)(- x, y) + e \big( \psi_{n_m} \big)(x, y) \Big) \, dx \, dy \leq M_p.
$$
This is sufficient to find a number $R_m \leq \widetilde{R}_m \leq 2 R_m$ satisfying~\eqref{eq:tranche}.

As a consequence of~\eqref{eq:tranche}, the restrictions $\psi_{n_m}(\pm \widetilde{R}_m, \cdot)$ lie in $H^1(\T_\ell)$. We now check that a function $\psi \in H^1(\T_\ell)$ with sufficiently small Ginzburg-Landau energy does not vanish.

\begin{step}
\label{S2}
Let $\psi \in H^1(\T_\ell)$. There exists a number $0 <\kappa_\ell < 1$, depending only on $\ell$, such that, if
\begin{equation}
\label{eq:fine}
\kappa(\psi) := \frac{1}{2 \ell} \int_{\T_\ell} |\partial_y \psi|^2 + \frac{1}{4 \ell} \int_{\T_\ell} \big( 1 - |\psi|^2 \big)^2 \leq \kappa_\ell,
\end{equation}
then the function $\psi$ does not vanish on $\T_\ell$, and can be lifted as $\psi = |\psi| \, e^{i \varphi}$, where $\varphi$ is a continuous, $\ell$-periodic, real-valued function. Moreover, there exists a number $C_\ell > 0$, depending only on $\ell$, such that the mean value
$$
\widehat{\psi} := \frac{1}{\ell} \int_{\T_\ell} \psi,
$$
satisfies
\begin{equation}
\label{eq:mayo}
\Big| \widehat{\psi} \Big| \geq 1 - C_\ell \, \kappa(\psi)^\frac{1}{2} \geq \frac{1}{2},
\end{equation}
and
\begin{equation}
\label{eq:thon}
\bigg| \textup{arg} \big( \widehat{\psi} \big) - \frac{1}{\ell} \int_{\T_\ell} \varphi \bigg| \leq C_\ell \, \kappa(\psi)^\frac{1}{2},
\end{equation}
when the function $\psi$ satisfies the condition in~\eqref{eq:fine}.
\end{step}

Assume that a function $\psi \in H^1(\T_\ell)$ vanishes. Up to a translation, we can assume that $\psi(0) = 0$, so that we can write
$$
\big| \psi(y) \big| \leq \sqrt{y} \, \| \partial_y \psi \|_{L^2(\T_\ell)}, 
$$
for any $0 \leq y \leq \ell$. In this case, either
$$
\frac{1}{2 \ell} \int_{\T_\ell} |\partial_y \psi|^2 \geq \frac{1}{4 \ell^2},
$$
or we can compute
$$
\frac{1}{4 \ell} \int_{\T_\ell} \big( 1 - |\psi|^2 \big)^2 \geq \frac{1}{4 \ell} \int_0^\ell \big( 1 - y \, \| \partial_y \psi \|_{L^2(\T_\ell)}^2 \big)^2 \, dy \geq \frac{1}{4} \big( 1 - \ell \, \| \partial_y \psi \|_{L^2(\T_\ell)}^2 \big)^2 \geq \frac{1}{16}.
$$
By contraposition, we conclude that a function $\psi$ such that
$$
\frac{1}{2 \ell} \int_{\T_\ell} |\partial_y \psi|^2 + \frac{1}{4 \ell} \int_{\T_\ell} \big( 1 - |\psi|^2 \big)^2 \leq \min \Big\{ \frac{1}{4 \ell^2}, \frac{1}{16} \Big\},
$$
does not vanish on $\T_\ell$.

Consider next its mean value $\widehat{\psi}$, and recall from the Poincar\'e-Wirtinger inequality that
\begin{equation}
\label{eq:pw}
\int_{\T_\ell} \big| \psi - \widehat{\psi} \big|^2 \leq \frac{\ell^2}{4 \pi^2} \int_{\T_\ell} \big| \partial_y \psi \big|^2,
\end{equation}
so that by the triangle and Cauchy-Schwarz inequalities,
\begin{align*}
\big| \widehat{\psi} \big| & \geq 1 - \frac{1}{\ell} \int_{\T_\ell} \big( 1 - |\psi| \big) - \frac{1}{\ell} \int_{\T_\ell} \big| \psi - \widehat{\psi} \big| \\
& \geq 1 - 2 \bigg( \frac{1}{4 \ell} \int_{\T_\ell} \big( 1 - |\psi|^2 \big)^2 \bigg)^\frac{1}{2} - \frac{\ell}{\sqrt{2} \pi} \bigg( \frac{1}{2 \ell} \int_{\T_\ell} |\partial_y \psi|^2 \bigg)^\frac{1}{2}.
\end{align*}
The first inequality in estimate~\eqref{eq:mayo} follows, when $C_\ell$ is supposed to be larger than $2 + \ell/(\sqrt{2} \pi)$, and it is sufficient to choose $\kappa_\ell$ less than $1/(4 \ell^2)$, $1/16$ and $1/(4 C_\ell)^2$, to obtain the second one.

We finally turn to~\eqref{eq:thon}. We first introduce a continuous, $\ell$-periodic, real-valued function $\varphi$ such that $\psi = |\psi| \, e^{i \varphi}$ on $\T_\ell$. We then compute, using the Poincar\'e-Wirtinger inequality,
$$
\Big| e^{i \varphi} - \frac{\widehat{\psi}}{|\widehat{\psi}|} \Big| = \Big| \frac{(|\widehat{\psi}| - |\psi|) \, \psi}{|\psi| \, |\widehat{\psi}|} + \frac{\psi - \widehat{\psi}}{|\widehat{\psi}|} \Big| \leq \frac{2 \, |\psi - \widehat{\psi}|}{|\widehat{\psi}|} \leq \big| \psi - \widehat{\psi} \big| \leq \int_{\T_\ell} \big| \partial_y \psi \big|,
$$
and deduce from the Cauchy-Schwarz inequality that
\begin{equation}
\label{eq:sesame}
\Big| e^{i \varphi} - \frac{\widehat{\psi}}{|\widehat{\psi}|} \Big| \leq \sqrt{2} \, \ell \, \kappa(\psi)^\frac{1}{2}.
\end{equation}
At this stage, we invoke the inequality
\begin{equation}
\label{eq:pignon}
\big| \textup{arg}(z_2) - \textup{arg}(z_1) \big| \leq \frac{2}{\pi} \big| z_1 - z_2 \big|,
\end{equation}
which holds for complex numbers $z_1$ and $z_2$ of modulus $1$, such that $|z_2 - z_1| \leq 1$. Note that the numbers $\textup{arg}(z_2)$ and $\textup{arg}(z_1)$ in this inequality are defined in $\R / 2 \pi \Z$, so that the notation $| \cdot |$ in the left-hand side of~\eqref{eq:pignon} corresponds to the natural distance on this set.

Going back to~\eqref{eq:sesame}, we finally fix $C_\ell := \max \{ 2 + \ell/(\sqrt{2} \pi), (2 \sqrt{2} \ell)/\pi \}$, so that $|e^{i \varphi} - \widehat{\psi}/|\widehat{\psi}|| \leq 1$, when $\kappa(\psi) \leq \kappa_\ell := \min \{ 1/(4 \ell^2), 1/16, 1/(4 C_\ell)^2 \}$. In this case, we infer from~\eqref{eq:sesame} and~\eqref{eq:pignon} that
$$
\Big| \varphi - \textup{arg} \big( \widehat{\psi} \big) \Big| \leq \frac{2 \, \sqrt{2} \, \ell}{\pi} \, \kappa(\psi)^\frac{1}{2} \leq C_\ell \, \kappa(\psi)^\frac{1}{2} \leq \frac{1}{2}.
$$
Invoking the continuity of the phase function $\varphi$, we can find a number $\tau \in \R$ such that $\textup{arg}(\widehat{\psi}) = \tau$ modulo $2 \pi$, and 
$$
\big| \varphi(x) - \tau \big| \leq C_\ell \, \kappa(\psi)^\frac{1}{2},
$$
for any $x \in \T_\ell$. Integrating this inequality on $\T_\ell$, we conclude that
$$
\bigg| \frac{1}{\ell} \int_{\T_\ell} \varphi - \tau \bigg| \leq C_\ell \, \kappa(\psi)^\frac{1}{2}.
$$
Since $\tau = \textup{arg}(\widehat{\psi})$ modulo $2 \pi$, this completes the proof of~\eqref{eq:thon}, as well as of Step~\ref{S2}.

Given a number $R \geq 2$, we now connect two functions $\psi_\pm \in H^1(\T_\ell)$ satisfying~\eqref{eq:fine} by a function $\Psi \in H^1([- R, R] \times \T_\ell, \C)$ of Ginzburg-Landau energy of order $\kappa(\psi_-) + \kappa(\psi_+) +1/R$. More precisely, we show

\begin{step}
\label{S3}
Let $\psi_\pm \in H^1(\T_\ell)$ satisfying~\eqref{eq:fine}. Given a number $R \geq 2$, there exists a function $\Psi \in H^1([-R, R] \times \T_\ell, \C)$ such that $\Psi(\pm R, \cdot) = \psi_\pm$, and moreover,
\begin{equation}
\label{eq:berne}
E_R \big( \Psi \big) \leq C_\ell \Big( \kappa(\psi_-) + \kappa(\psi_+) + \frac{1}{R} \Big), \quad \text{ and } \quad \Big| [P_R] \big( \Psi \big) \Big| \leq C_\ell \Big( \kappa(\psi_-) + \kappa(\psi_+) \Big)^\frac{1}{2}, 
\end{equation}
for a number $C_\ell \geq 0$, depending only on $\ell$.
\end{step}

Under condition~\eqref{eq:fine}, it follows from Step~\ref{S2} that the functions $\psi_\pm$ do not vanish. In particular, we can find continuous, $\ell$-periodic, real-valued functions $\rho_\pm$ and $\varphi_\pm$ such that $\psi_\pm = \rho_\pm \, e^{i \varphi_\pm}$ on $\T_\ell$. Using this notation, we observe that
$$
\kappa(\psi_\pm) = \frac{1}{2 \ell} \int_{\T_\ell} \big( (\partial_y \rho_\pm)^2 + \rho_\pm^2 (\partial_y \varphi_\pm)^2 \big) + \frac{1}{4 \ell} \int_{\T_\ell} \big( 1 - \rho_\pm^2 \big)^2.
$$
Setting
$$
\widehat{\varphi}_\pm = \frac{1}{\ell} \int_{\T_\ell} \varphi_\pm(y) \, dy,
$$
we first derive from the Sobolev embedding theorem and the Poincar\'e-Wirtinger inequality the existence of a number $C_\ell$, depending only on $\ell$, such that
\begin{equation}
\label{eq:emmental}
\big\| 1 - \rho_\pm \big\|_{L^\infty(\T_\ell)}^2 \leq C_\ell \, \kappa(\psi_\pm), \quad \text{ and } \quad \big\| \varphi_\pm - \widehat{\varphi}_\pm \big\|_{L^2(\T_\ell)} \leq \frac{\ell}{2 \pi} \, \big\| \partial_y \varphi_\pm \big\|_{L^2(\T_\ell)}.
\end{equation}
Going back to~\eqref{eq:fine}, we note that the modulus $\rho_\pm$ are uniformly bounded by a number $C_\ell$, depending only on $\ell$.

We next define the map $\Psi$ by introducing affine interpolations of the modulus $\rho_\pm$ and the phases $\varphi_\pm$. More precisely, we set
\begin{equation}
\label{def:rho}
\rho(x, y) = \begin{cases} 1, \quad \text{ for } |x| \leq R - 1 \text{ and } y \in \T_\ell, \\ 1 + \big( R - 1 - |x| \big) \big( 1 - \rho_\pm(y) \big), \quad \text{ for } R - 1 \leq \pm \, x \leq R \text{ and } y \in \T_\ell. \end{cases}
\end{equation}
Introducing the integer $k \in \Z$ such that $2 k \pi \leq \widehat{\varphi}_+ - \widehat{\varphi}_- < 2 (k + 1) \pi$, we also set
\begin{equation}
\label{def:varphi}
\varphi(x, y) = \begin{cases} \frac{\widehat{\varphi}_+ - \widehat{\varphi}_- - 2 k \pi}{2 (R - 1)} \, x + \frac{\widehat{\varphi}_+ + \widehat{\varphi}_- + 2 k \pi}{2}, \quad \text{ for } |x| \leq R - 1 \text{ and } y \in \T_\ell, \\ \varphi_\pm(y) + \big( R - |x| \big) \big( \widehat{\varphi_\pm} - \varphi_\pm(y) \big), \quad \text{ for } R - 1 \leq \pm \, x \leq R \text{ and } y \in \T_\ell. \end{cases} 
\end{equation}
The map $\Psi := \rho \, e^{i \varphi}$ then lies in $H^1([-R, R] \times \T_\ell)$, with $\Psi(\pm R, \cdot) = \psi_\pm$. From~\eqref{eq:emmental} and explicit computations, we first derive the estimate for $E_R(\Psi)$ in~\eqref{eq:berne}.

Concerning the estimate for the momentum $[P_R](\Psi)$, we write $\Psi =
\widehat{\Psi} + W$ according to~\eqref{eq:decomp1d}, and recall from~\eqref{eq:mayo} that the mean values $\widehat{\Psi}(\pm R) = \widehat{\psi}_\pm$ are larger than $1/2$. In view of~\eqref{def:P_R}, the momentum $[P_R](\Psi)$ is then well-defined by
$$
[P_R] \big( \Psi \big) = \frac{1}{2} \int_{- R}^R \bigg( \big\langle i \partial_x \widehat{\Psi}, \widehat{\Psi} \big\rangle_\C + \frac{1}{\ell} \int_{\T_\ell} \big\langle i \partial_x W, W \big\rangle_\C \bigg) + \frac{1}{2} \textup{arg}\big( \widehat{\Psi}(R) \big) - \frac{1}{2} \textup{arg} \big( \widehat{\Psi}(- R) \big) \quad \text{ modulo } \pi.
$$
Recall that $\widehat{\Psi}$ only depends on $x$ and that $\int_{\T_\ell} W(x,y)
\, dy = 0$ for all $x \in \R.$ By definition of the function $\Psi$, we can
therefore rewrite the previous formula as
$$
[P_R] \big( \Psi \big) = \frac{1}{2 \ell} \int_{- R}^R \int_{\T_\ell} \big\langle i \partial_x \Psi, \Psi \big\rangle_\C + \frac{1}{2} \textup{arg} \big( \widehat{\psi}_+ \big) - \frac{1}{2} \textup{arg} \big( \widehat{\psi}_- \big) \quad \text{ modulo } \pi.
$$
Since $\langle i \partial_x \Psi, \Psi \rangle_\C = - \rho^2 \, \partial_x \varphi$, we are led to 
\begin{align*}
[P_R] \big( \Psi \big) = & \frac{1}{2 \ell} \int_{- R}^R \int_{\T_\ell} \big( 1 - \rho^2 \big) \, \partial_x \varphi + \frac{1}{2} \bigg( \textup{arg} \big( \widehat{\psi}_+ \big) - \frac{1}{\ell} \int_{\T_\ell} \varphi_+(y) \, dy \bigg) \\
& - \frac{1}{2} \bigg( \textup{arg} \big( \widehat{\psi}_- \big) - \frac{1}{\ell} \int_{\T_\ell} \varphi_-(y) \, dy \bigg) \quad \text{ modulo } \pi.
\end{align*}
In view of~\eqref{eq:emmental},~\eqref{def:rho} and~\eqref{def:varphi}, we can find further numbers $C_\ell > 0$, depending only on $\ell$, such that the first term in the right-hand side of this inequality satisfies
$$
\bigg| \frac{1}{2 \ell} \int_{- R}^R \int_{\T_\ell} \big( 1 - \rho^2 \big) \, \partial_x \varphi \bigg| \leq C_\ell \int_{\T_\ell} \Big( \big| 1 - \rho_+^2 \big| \, \big| \varphi_+ - \widehat{\varphi}_+ \big| + \big| 1 - \rho_-^2 \big| \, \big| \varphi_- - \widehat{\varphi}_- \big| \Big) \leq C_\ell \Big( \kappa(\psi_+) + \kappa(\psi_-) \Big).
$$
Concerning the second and third terms, we directly invoke~\eqref{eq:thon} to obtain
$$
\bigg| \textup{arg} \big( \widehat{\psi}_\pm \big) - \frac{1}{\ell} \int_{\T_\ell} \varphi_\pm(y) \, dy \bigg| \leq C_\ell \kappa(\psi_\pm)^\frac{1}{2}.
$$
Gathering the two previous estimates, we conclude that
$$
\Big| [P_R] \big( \Psi \big) \Big| \leq C_\ell \Big( \kappa(\psi_+) + \kappa(\psi_-) \Big)^\frac{1}{2},
$$
due to the bounds $\kappa(\psi_+) \leq \kappa_\ell$ and $\kappa(\psi_-) \leq \kappa_\ell$. This is exactly the second inequality in~\eqref{eq:berne}.

We now collect the three previous steps in order to complete the proof of Lemma~\ref{lem:vaud}.

\begin{step}
\label{S4}
Conclusion.
\end{step}

Fix an integer $m \geq 1$. In view of Step~\ref{S1}, we can find a number $R_m \leq \widetilde{R}_m \leq 2 R_m$, such that~\eqref{eq:tranche} holds for some number $M_p$. Up to a subsequence, we can assume that $M_p/R_m \leq \kappa_\ell$. By~\eqref{eq:tranche}, the maps $\psi_\pm := \psi_{n_m}(\pm \widetilde{R}_m, \cdot)$ then satisfy the condition in~\eqref{eq:fine}. Observing that $\widetilde{R}_m \geq R_m \geq 2$, we can invoke Step~\ref{S3} with $R = \widetilde{R}_m$ to find a function $\widetilde{\psi}_{n_m} := \Psi \in H^1([- \widetilde{R}_m, \widetilde{R}_m] \times \T_\ell, \C)$ such that
$$
\widetilde{\psi}_{n_m} \big( \pm \widetilde{R}_m, \cdot \big) = \psi_{n_m} \big( \pm \widetilde{R}_m, \cdot \big),
$$
$$
E_{\widetilde{R}_m} \big( \widetilde{\psi}_{n_m} \big) \leq C_\ell \bigg( \kappa \Big( \psi_{n_m}(\widetilde{R}_m, \cdot) \Big) + \kappa \Big( \psi_{n_m}(- \widetilde{R}_m, \cdot) \Big) + \frac{1}{\widetilde{R}_m} \bigg),
$$
and 
$$
\Big| [P_{\widetilde{R}_m}] \big( \widetilde{\psi}_{n_m} \big) \Big| \leq C_\ell \Big( \kappa \big( \psi_{n_m}(\widetilde{R}_m, \cdot) \big) + \kappa \big( \psi_{n_m}(- \widetilde{R}_m, \cdot) \big) \Big)^\frac{1}{2}.
$$
The two estimates in Lemma~\ref{lem:vaud} then follow from~\eqref{eq:tranche} and the fact that $\widetilde{R}_m \geq R_m$. This concludes the proof of this lemma. \qed

%%%%%%%%%%%%%%%%%%%%%%%%%%%%%%%
%%%%%%%%%%%%%%%%%%%%%%%%%%%%%%%
\subsection{Proof of Proposition~\ref{prop:tessin}}
%%%%%%%%%%%%%%%%%%%%%%%%%%%%%%%
%%%%%%%%%%%%%%%%%%%%%%%%%%%%%%%

For a fixed integer $m$, we decompose the energy $E(\psi_{n_m})$ as
$$
E \big( \psi_{n_m} \big) = E_{\widetilde{R}_m} \big( \psi_{n_m} \big) + E_{\widetilde{R}_m^c} \big( \psi_{n_m} \big).
$$
In view of~\eqref{eq:massezero},~\eqref{eq:uri} and~\eqref{eq:grisons}, we have on the one hand,
$$
E_{\widetilde{R}_m} \big( \psi_{n_m} \big) \geq E_{R_m} \big( \psi_{n_m} \big) \geq \max \Big\{ \delta_\infty, E_{R_m} \big( \psi_\infty \big) - \frac{1}{m} \Big\} \geq \max \Big\{ \delta_\infty, E \big( \psi_\infty \big) - \frac{2}{m} \Big\}.
$$
On the other hand, we infer from Lemma~\ref{lem:vaud} that
$$
E_{\widetilde{R}_m^c} \big( \psi_{n_m} \big) = E_{\widetilde{R}_m^c} \big( \widetilde{\psi}_{n_m} \big) \geq E \big( \widetilde{\psi}_{n_m} \big) - \frac{C_\ell}{R_m},
$$
where $C_\ell$ is the number in Lemma~\ref{lem:vaud}. Since $R_m \to + \infty$ as $m \to \infty$, inequality~\eqref{eq:tessin0} follows by summation.

We argue similarly for the momentum, which we decompose as
$$
[P] \big( \psi_{n_m} \big) = [P_{\widetilde{R}_m}] \big( \psi_{n_m} \big) + [P_{\widetilde{R}_m^c}] \big( \psi_{n_m} \big).
$$
Using~\eqref{eq:uri} and~\eqref{eq:grisons}, we obtain
$$
\big| [P_{\widetilde{R}_m}] \big( \psi_{n_m} \big) - [P] \big( \psi_\infty \big) \big| \leq \big| [P_{\widetilde{R}_m}] \big( \psi_{n_m} \big) - [P_{\widetilde{R}_m}] \big( \psi_\infty \big) \big| + \big| [P_{\widetilde{R}_m^c}] \big( \psi_\infty \big) \big| \leq \frac{2}{m},
$$
while, by Lemma~\ref{lem:vaud},
$$
\big| [P_{\widetilde{R}_m^c}] \big( \psi_{n_m} \big) - [P] \big( \widetilde{\psi}_{n_m} \big) \big| \leq \big| [P_{\widetilde{R}_m}] \big( \widetilde{\psi}_{n_m} \big) \big| \leq \frac{C_\ell}{R_m^\frac{1}{2}}.
$$
Inequality~\eqref{eq:tessin1} follows as well by summation. This completes the proof of Proposition~\ref{prop:tessin}. \qed

%%%%%%%%%%%%%%%%%%%%%%%%%%%%%%%%%%%%%%%%%%
%%%%%%%%%%%%%%%%%%%%%%%%%%%%%%%%%%%%%%%%%%
%%%%%%%%%%%%%%%%%%%%%%%%%%%%%%%%%%%%%%%%%%
\appendix
\section{Properties of the minimizing energy in one space dimension}
\label{sec:min-1d}
%%%%%%%%%%%%%%%%%%%%%%%%%%%%%%%%%%%%%%%%%%
%%%%%%%%%%%%%%%%%%%%%%%%%%%%%%%%%%%%%%%%%%
%%%%%%%%%%%%%%%%%%%%%%%%%%%%%%%%%%%%%%%%%%

In this first appendix, we recall properties of the minimizing energy $\boI_\textup{1d}$, which are useful for the proofs of Lemmas~\ref{lem:cond-concave} and~\ref{lem:not-1d}.

\begin{lem}
\label{lem:prop-gI}
The minimizing energy $\boI_\textup{1d}$ is a $\pi$-periodic and even function on $\R$ given by
$$
\boI_\textup{1d}(q) = \frac{1}{3} \big( 2 - \gc_q^2 \big)^\frac{3}{2},
$$
for any number $q \in \R$. The speed $\gc_q$ in this formula is the unique number in $(- \sqrt{2}, \sqrt{2}]$ such that
$$
[q] = \Xi(\gc_q) := \frac{\pi}{2} - \arctan \bigg( \frac{\gc_q}{\sqrt{2 - \gc_q^2}} \bigg) - \frac{\gc_q}{2} \sqrt{2 - \gc_q^2} \quad \text{ modulo } \pi.
$$
In particular, $\boI_\textup{1d}$ is $\sqrt{2}$-Lipschitz on $\R$, and concave on the interval $[0, \pi]$.
\end{lem}

\begin{proof}
Note first that the minimizing energy $\boI_\textup{1d}$ is by definition a $\pi$-periodic function on $\R$. It is also even due to the property that $[P](\overline{\psi}) = - [P](\overline{\psi})$ for any function $\psi \in X(\R)$. When $0 \leq q < \pi/2$, the formula for $\boI_\textup{1d}(q)$ was derived in~\cite[Proposition A.6]{deLGrSm1} from~\cite{BetGrSa2}. It extends to $q = \pi/2$ by~\cite[Lemma 6]{BeGrSaS1}. Since $\Xi(- c) = \pi - \Xi(c) = - \Xi(c)$ modulo $\pi$, it also extends to $[- \pi/2, 0]$ by parity, and then to $\R$ by periodicity.

Observe next that
$$
\Xi'(c) = - (2 - c^2)^\frac{1}{2},
$$
for $- \sqrt{2} < c < \sqrt{2}$. As a consequence, the function $\Xi$ is a smooth diffeomorphism from $(- \sqrt{2}, \sqrt{2})$ to $(0, \pi)$, that extends into a continuous function from $[- \sqrt{2}, \sqrt{2}]$ to $[0, \pi]$. In particular, the map $q \to \gc_q$ is well-defined and continuous on $[0, \pi]$, and smooth from $(0, \pi)$ to $(- \sqrt{2}, \sqrt{2})$, with
\begin{equation}
\label{eq:deriv-gc}
\frac{d \gc_q}{dq} = - \frac{1}{(2 - \gc_q^2)^\frac{1}{2}}.
\end{equation}
Therefore, the function $\boI_\textup{1d}$ is also continuous on $[0, \pi]$ and smooth on $(0, \pi)$. Moreover, we check that
\begin{equation}
\label{eq:deriv-gI}
\boI_\textup{1d}'(q) = \gc_q.
\end{equation}
Since $|\gc_q| < \sqrt{2}$ for any $0 < q < \pi$, the function $\boI_\textup{1d}$ is $\sqrt{2}$-Lipschitz on the interval $(0, \pi)$. By continuity, this property extends to $[0, \pi]$, and again by periodicity, to $\R$. 

Going back to~\eqref{eq:deriv-gc}, we next notice that the map $q \to \gc_q$ is decreasing on $(0, \pi)$. In view of~\eqref{eq:deriv-gI}, so is the derivative $\boI_\textup{1d}'$. This guarantees the concavity of the function $\boI_\textup{1d}$ on the interval $[0, \pi]$, and completes the proof of Lemma~\ref{lem:prop-gI}. 
\end{proof}

%%%%%%%%%%%%%%%%%%%%%%%%%%%%%%%%%%%%%%%%%%%
%%%%%%%%%%%%%%%%%%%%%%%%%%%%%%%%%%%%%%%%%%%
%%%%%%%%%%%%%%%%%%%%%%%%%%%%%%%%%%%%%%%%%%%
\section{Useful characterizations of concavity and strict sub-additivity}
\label{sec:concave}
%%%%%%%%%%%%%%%%%%%%%%%%%%%%%%%%%%%%%%%%%%%
%%%%%%%%%%%%%%%%%%%%%%%%%%%%%%%%%%%%%%%%%%%
%%%%%%%%%%%%%%%%%%%%%%%%%%%%%%%%%%%%%%%%%%%

In this second appendix, we first give the

\begin{proof}[Proof of Lemma~\ref{lem:concave}]
We assume for the sake of a contradiction that the function $f$ is not concave on $[a, b]$. We are then able to find three numbers $a \leq \alpha < \beta \leq b$ and $0 < \mu < 1$ such that
$$
f(x_\mu) < \mu f(\alpha) + (1 - \mu) f(\beta),
$$
with $x_\mu = \mu \alpha + (1 - \mu) \beta$. We consider the straight line $\Delta$ going through the points $(\alpha, f(\alpha))$ and $(\beta, f(\beta))$, whose equation is explicitly given by $y = \tau x + \sigma$, with $\tau = (f(\beta) - f(\alpha))/(\beta - \alpha)$ and $\sigma = (\beta f(\alpha) - \alpha f(\beta))/(\beta - \alpha)$. We can rewrite the previous inequality as the fact that the straight line $\Delta_\mu$, which is parallel to $\Delta$ and goes through the point $(x_\mu, f(x_\mu))$, is strictly below the line $\Delta$. Indeed, we check that
\begin{equation}
\label{don}
f(x_\mu) - \tau x_\mu = f(x_\mu) - \beta \frac{f(\beta) - f(\alpha)}{\beta - \alpha} + \mu \big( f(\beta) - f(\alpha) \big) < f(\beta) - \beta \frac{f(\beta) - f(\alpha)}{\beta - \alpha} = \sigma.
\end{equation}
As a consequence, we can introduce the straight line $\Delta_*$, which contains at least one point $(x, f(x))$ for a number $\alpha \leq x \leq \beta$, and which is, among all the parallel lines to $\Delta$, the most below this line. The line $\Delta_*$ is given by the equation $y = \tau x + \sigma_*$, in which
$$
\sigma_* = \min \big\{ f(x) - \tau x, x \in [\alpha, \beta] \big\}.
$$
This number is well-defined by continuity of the function $f$ and it is strictly less than $\sigma$ by~\eqref{don}. We finally introduce the number
$$
x_* = \min \big\{ x \in [\alpha, \beta] \text{ s.t. } f(x) = \tau x + \sigma_* \big\}.
$$
Again by continuity of the function $f$, as well as by definition of the number $\sigma_*$, the number $x_*$ is well-defined. Moreover, it lies in the interval $(\alpha, \beta)$ due to the fact that $\sigma_* < \sigma$. Therefore, we can deduce from~\eqref{eq:cond-concave} the existence of a number $\delta_*$ such that
$$
\frac{1}{2} \Big( f(x_* + \delta_*) + f(x_* - \delta_*) \Big) \leq f(x_*),
$$
where we can assume that $\alpha \leq x_* - \delta_* < x_* + \delta_* \leq \beta$. In this case, it follows from the definition of the number $x_*$ that
$$
f(x_*) = \tau x_* + \sigma_*, \quad \text{ and } \quad f(x_* - \delta_*) > \tau (x_* - \delta_*) + \sigma_*,
$$
so that
$$
f(x_* + \delta_*) < \tau (x + \delta_*) + \sigma_*.
$$
Since $\alpha \leq x_* + \delta_* \leq \beta$, this is a contradiction with the definition of the number $\sigma_*$. Hence the concavity of the function $f$ on $[a , b]$ is proved.
\end{proof}

We next provide the

\begin{proof}[Proof of Lemma~\ref{lem:strict-subadditivity}]
Using the concavity of the function $f$ on $[0, R]$, and the fact that $f(0) = 0$, we already know that this function is sub-additive on $[0, R]$. Recall indeed that
\begin{equation}
\label{eq:conc-sub-add}
x \, f(y) \geq y \, f(x) + (x - y) \, f(0) = y \, f(x),
\end{equation}
for any numbers $0 \leq y \leq x \leq R$, so that
\begin{equation}
\label{eq:sub-add}
f(x_1) + f(x_2) \geq f(x_1 + x_2), 
\end{equation}
for any numbers $0 \leq x_1, x_2 \leq R$, with $0 \leq x_1 + x_2 \leq R$.

We now assume that the function $f$ is not strictly sub-additive on $[0, R]$. In this situation, there exist numbers $0 < x_1, x_2 \leq R$, with $x_* := x_1 + x_2 \in (0, R]$, such that the previous inequality is an equality. We claim that
\begin{equation}
\label{eq:affine}
f(y) = \frac{y}{x_*} \, f(x_*),
\end{equation}
for any $0 \leq y \leq x_*$. This identity already holds for $y = 0$ and $y = x_*$. Moreover, it follows from~\eqref{eq:conc-sub-add} that the case of equality into~\eqref{eq:sub-add} can only hold if and only if
$$
f(x_1) = \frac{x_1}{x_*} \, f(x_*) \quad \Big(\text{and } \quad f(x_2) = \frac{x_2}{x_*} \, f(x_*) \Big). 
$$
When $x_1 < y < x_*$, we again deduce from the concavity of the function $f$ that
$$
\frac{y}{x_*} \, f(x_*) = \frac{y}{x_1} \, f(x_1) \geq f(y) + \frac{y - x_1}{x_1} \, f(0) = f(y),
$$
and~\eqref{eq:affine} follows from the already proved reverse inequality~\eqref{eq:conc-sub-add}. Similarly, when $0 < y < x_1$, we have
$$
\frac{x_1}{x_*} f(x_*) = f(x_1) \geq \frac{x_* - x_1}{x_* - y} f(y) + \frac{x_1 - y}{x_* - y} f(x_*),
$$
which amounts to the inequality
$$
\frac{y}{x_*} \, f(x_*) \geq f(y).
$$
We are again led to~\eqref{eq:affine} by the reverse inequality~\eqref{eq:conc-sub-add}. This completes the proof of~\eqref{eq:affine}.

It is then enough to set $\mu = f(x_*)/x_*$ to obtain~\eqref{eq:linear}. Since the function $f$ is non-negative, the number $\mu$ is also non-negative. This concludes the proof of the alternative in Lemma~\ref{lem:strict-subadditivity}.
\end{proof}

\begin{merci}
P. Gravejat is very grateful to Laboratory Jacques-Louis Lions (UMR CNRS 7598) in Sorbonne University for its warm hospitality during the completion of this work. 

The authors acknowledge support from the project ``Dispersive and random waves'' (ANR-18-CE40-0020-01) of the Agence Nationale de la Recherche. A.~de Laire was also supported by the Labex CEMPI (ANR-11-LABX-0007-01), and P.~Gravejat, by the CY Initiative of Excellence (Grant ``Investissements d'Avenir'' ANR-16-IDEX-0008).
\end{merci}

\bibliographystyle{plain}
%\bibliography{Bibliogr}

\end{document}